\documentclass[12pt]{amsart}

\usepackage{amsmath}
\usepackage{amsmath,amsfonts,amssymb,amsthm}
\usepackage{graphicx,color}

\usepackage{hyperref}
\DeclareMathRadical{\sqrtsign}{symbols}{"70}{largesymbols}{"70}
\providecommand{\abs}[1]{\lvert#1\rvert}

\newlength{\figboxwidth}             
\newcommand{\grad}{\nabla}
\newcommand{\infinity}{\infty}

\def\@ifundefined#1#2#3%
  {\expandafter\ifx\csname#1\endcsname\relax#2\else#3\fi}

\@ifundefined{theoremstyle}{
}{
\theoremstyle{plain} 
}
\newtheorem{theorem}{Theorem}[section]

\newtheorem{proposition}[theorem]{Proposition}
\newtheorem{lemma}[theorem]{Lemma}

\newtheorem{corollary}[theorem]{Corollary}

\@ifundefined{theoremstyle}{
}{
\theoremstyle{definition} 
}
\newtheorem{definition}[theorem]{Definition}

\newtheorem{remark}[theorem]{Remark}

\mathchardef\GG="321D
\newcommand{\mcc}[1]{{}}

\numberwithin{equation}{section}

\title[Concentration of dimension in the Lagrange spectrum]
{Concentration of dimension in the Lagrange spectrum}

\author{Christian Camilo Silva Villamil}

\address{Christian Camilo Silva Villamil: SUSTech International Center for Mathematics, P. R. China.
}
\email{ccsilvav@sustech.edu.cn}


\begin{document}

\begin{abstract}
Let $\varphi$ be a smooth conservative diffeomorphism of a compact surface $S$ and let $\Lambda$ be a mixing horseshoe of $\varphi$. Given a smooth real function $f$ defined on $S$, we define for points $\eta$ in the unstable Cantor set of the pair $(\varphi,\Lambda)$, a generalization, $k_{\varphi,\Lambda,f}(\eta)$, of the best constant of Diophantine approximation for irrational numbers. We study the set of points $\eta$ for which the sets $k_{\varphi,\Lambda,f}^{-1}((-\infty,\eta])$ and $k_{\varphi,\Lambda,f}^{-1}(\eta)$ have the same Hausdorff dimension and when the Hausdorff dimension of $\Lambda$ is less than one, we describe generically the local Hausdorff dimension of the dynamical Lagrange spectrum, $\mathcal{L}_{\varphi,\Lambda,f}$, restricted to this set of points. Finally, we recover the same results for the classical Lagrange spectra.     

\end{abstract}

\maketitle

\tableofcontents

\section{Introduction}
\subsection{Classical spectra}\label{ss.classical-Markov-Lagrange}
The classical Lagrange and Markov spectra are closed subsets of the real line related to Diophantine approximations. They arise naturally in the study of rational approximations of irrational numbers and of indefinite binary quadratic forms, respectively. 

Given $\alpha\in\mathbb R\setminus\mathbb Q$, set
\begin{eqnarray*}
	k(\alpha)&=&\sup \left\{k>0:\left|\alpha -\frac{p}{q}\right|<\frac{1}{kq^2} \ \mbox{has infinitely many rational solution} \ \frac{p}{q}\right \}\\&
	=&\limsup_{p\in \mathbb{Z},q\in \mathbb{N}, p,q\to \infty}|q(q\alpha-p)|^{-1}\in \mathbb{R}\cup \{\infty\}
\end{eqnarray*}
for the best constant of Diophantine approximations of $\alpha$. The {\it classical Lagrange spectrum} is the set 
$$\mathcal{L}=\{k(\alpha): \alpha \in \mathbb{R}\setminus\mathbb{Q}\ \text{and}\ k(\alpha)<\infty\}.$$
The study of the structure of $\mathcal{L}$ is a classical subject which began with
Markov that showed in \cite{M79} that $\mathcal{L}\cap(-\infty, 3)=\{\sqrt{9-4/z_n^2}:n\in \mathbb{N}\}$ where $z_n$ are the \emph{Markov numbers}, that is, the largest coordinate of a triple $(x_n,y_n,z_n)\in \mathbb{N}^3$ verifying the Markov equation $x_n^2+y_n^2+z_n^2=3x_ny_nz_n.$ And then that $3$ is the first accumulation point of the Lagrange spectrum.

Similarly, given a real quadratic form $q(x,y)=ax^2+bxy+cy^2$, let $\Delta(q)=b^2-4ac$ its discriminant. Another interesting set is the {\it classical Markov spectrum}, defined by
$$\mathcal{M}:=\left\{\frac{\sqrt{\Delta(q)}}{\inf\limits_{(x,y)\in\mathbb{Z}^2\setminus\{(0,0)\}} |q(x,y)|} < \infty: q\ \text{is indefinite and}\ \Delta(q)>0\right\}.$$ 

Both the Lagrange and Markov spectra have a dynamical interpretation given by Perron in \cite{P}. This fact is an important motivation for our work: Given a bi-infinite sequence $\theta=(\theta_n)_{n\in\mathbb{Z}}\in\mathbb{N}^{\mathbb{Z}}$, let $\sigma(\theta) = (\theta_{n+1})_{n\in\mathbb{Z}}$ and 
$$f(\theta)=[0;\theta_1,\theta_2,\dots]+\theta_0+[0;\theta_{-1}, \theta_{-2},\dots].$$ 
If the Markov value $m(\theta)$ of $\theta$ is $m(\theta)=\sup\limits_{i\in\mathbb{Z}} f(\sigma^i(\theta))$ and the Lagrange value $\ell(\theta)$ is $\ell(\theta)=\limsup \limits_{i\to \infty} f(\sigma^i(\theta)).$ Then the Lagrange spectrum is the set
$$\mathcal{L}=\{\ell(\theta)<\infty: \theta\in\mathbb{N}^{\mathbb{Z}}\}$$
and the Markov spectrum is the set
$$\mathcal{M}=\{m(\theta)<\infty: \theta\in\mathbb{N}^{\mathbb{Z}}\}.$$

It follows from these characterizations that $\mathcal{L}$ and $\mathcal{M}$ are closed subsets of $\mathbb R$ and that $\mathcal{L}\subset \mathcal{M}$. This last interpretation, in terms of $\sigma$ and $f$, of the Lagrange and Markov spectra, admits a natural generalization in the context of hyperbolic dynamics as we will see in the next section.

Moreira in \cite{M3} proved several results on the geometry of the Markov and Lagrange spectra, for example, that the map $D:\mathbb{R} \rightarrow [0,1)$, given by
$$D(\eta)=HD(k^{-1}(-\infty,\eta))=HD(k^{-1}(-\infty,\eta])$$
is continuous, surjective and $\max\{t\in \mathbb{R}:D(t)=0\}=3$. Also that 
\begin{equation}\label{e2}
L(\eta):=HD(\mathcal{L}\cap(-\infty,\eta))= HD(\mathcal{M}\cap(-\infty,\eta))=\min \{1,2D(\eta)\}
\end{equation}
and
$$\lim_{\eta\rightarrow \infty}HD(k^{-1}(\eta))=1.$$
In the mentioned paper, it is also asked if the restriction of the function $L^{\text{loc}}:\mathcal{L}\rightarrow \mathbb{R}$ given by $L^{\text{loc}}(t)=\lim \limits_{\epsilon\to 0^{+}}HD(\mathcal{L}\cap (t-\epsilon,t+\epsilon))$ to $\mathcal{L}^{'}$ is non-decreasing.
\subsection{Dynamical spectra}
Let $\varphi:S\rightarrow S$ be a diffeomorphism of a $C^{\infty}$ compact surface $S$ with a mixing horseshoe $\Lambda$ and let $f:S\rightarrow \mathbb{R}$ be a differentiable function. Following the above characterization of the classical spectra, we define the maps $\ell_{\varphi,f}: \Lambda \rightarrow \mathbb{R}$ and $m_{\varphi,f}: \Lambda \rightarrow \mathbb{R}$ given by $\ell_{\varphi,f}(x)=\limsup\limits_{n\to \infty}f(\varphi^n(x))$ and $m_{\varphi,f}(x)=\sup\limits_{n\in\mathbb{Z}}f(\varphi^n(x))$ for $x\in \Lambda$ and call $\ell_{\varphi,f}(x)$ the \textit{Lagrange value} of $x$ associated to $f$ and $\varphi$ and also $m_{\varphi,f}(x)$ the \textit{Markov value} of $x$ associated to $f$ and $\varphi$. The sets
$$\mathcal{L}_{\varphi,\Lambda,f}=\ell_{\varphi,f}(\Lambda)=\{\ell_{\varphi,f}(x):x\in \Lambda\}$$
and
$$\mathcal{M}_{\varphi,\Lambda, f}=m_{\varphi,f}(\Lambda)=\{m_{\varphi,f}(x):x\in \Lambda\}$$
are called \textit{Lagrange Spectrum} of $(\varphi,\Lambda,f)$ and \textit{Markov Spectrum} of $(\varphi,\Lambda,f)$.

In order to announce our main theorem, let us first fix a Markov partition $\{R_a\}_{a\in \mathcal{A}}$ with sufficiently small diameter consisting of rectangles $R_a \sim I_a^u \times I_a^s$ delimited by compact pieces $I_a^s$, $I_a^u$, of stable and unstable manifolds of certain points of $\Lambda$. The set $\mathcal{B}\subset \mathcal{A}^{2}$ of admissible transitions consist of pairs $(a,b)$ such that $\varphi(R_a)\cap R_{b}\neq \emptyset$; so, we can define the transition matrix $B$ by
$$b_{ab}=1 \ \ \text{if} \ \  \varphi(R_a)\cap R_b\neq \emptyset \ \ \text{and}  \ b_{ab}=0 \ \ \ \text{otherwise, for $(a,b)\in \mathcal{A}^{2}$.}$$

Let $\Sigma_{\mathcal{A}}=\left\{\underline{a}=(a_{n})_{n\in \mathbb{Z}}:a_{n}\in \mathcal{A} \ \text{for all} \ n\in \mathbb{Z}\right\}$ and consider the homeomorphism of $\Sigma_{\mathcal{A}}$, the shift, $\sigma:\Sigma_{\mathcal{A}}\to\Sigma_{\mathcal{A}}$ defined by $\sigma(\underline{a})_{n}=a_{n+1}$. Let $\Sigma_{\mathcal{B}}=\left\{\underline{a}\in \Sigma_{\mathcal{A}}:b_{a_{n}a_{n+1}}=1\right\}$, this set is closed and $\sigma$-invariant subspace of $\Sigma_{\mathcal{A}}$. Still denote by $\sigma$ the restriction of $\sigma$ to $\Sigma_{\mathcal{B}}$, the pair $(\Sigma_{\mathcal{B}},\sigma)$ is a subshift of finite type, see \cite{{Shub}} chapter 10. The dynamics of $\varphi$ on $\Lambda$ is topologically conjugate to the sub-shift $\Sigma_{\mathcal{B}}$, namely, there is a homeomorphism $\Pi: \Lambda \to \Sigma_{\mathcal{B}}$ such that $\varphi\circ \Pi=\Pi\circ \sigma$.

Recall that the stable and unstable manifolds of $\Lambda$ can be extended to locally invariant $C^{1+\alpha}$ foliations in a neighborhood of $\Lambda$ for some $\alpha>0$. Using these foliations it is possible define projections $\pi^u_a: R_a\rightarrow I^s_a\times \{i^u_a\}$ and $\pi^s_a: R_a\rightarrow \{i^s_a\} \times I^u_a$ of the rectangles into the connected components $I^s_a\times \{i^u_a\}$ and $\{i^s_a\} \times I^u_a$ of the stable and unstable boundaries of $R_a$, where $i^u_a\in \partial I^u_a$ and $i^s_a\in \partial I^s_a$ are fixed arbitrarily. 

Set $\pi^s(x)=\pi^u_a(x)$ and $\pi^u(x)=\pi^s_a(x)$ if $x\in R_a$. In this way, we have the stable and unstable Cantor sets
$$K^s:=\pi^s(\Lambda)=\bigcup_{a\in \mathcal{A}}\pi^u_a(\Lambda\cap R_a)$$ 
and
$$K^u:=\pi^u(\Lambda)=\bigcup_{a\in \mathcal{A}}\pi^s_a(\Lambda\cap R_a).$$
In fact $K^s$ and $K^u$ are $C^{1+\alpha}$ dynamically defined, associated to the expanding maps $\psi_s$ and $\psi_u$ defined by
$$\psi_s(\pi^s(y))=\pi^s(\varphi^{-1}(y))$$
and 
$$\psi_u(\pi^u(z))=\pi^u(\varphi(z)).$$
Usually, we will consider some subsets of $\Lambda$ through its projections on the unstable Cantor set $K^u$. Indeed, we will consider $\pi^u(X)$, where $X\subset \Lambda$ is compact and $\varphi$-invariant. 

Note that if $x_1,x_2\in \Lambda$ are such that $\pi^u(x_1)=\pi^u(x_2)$, then $x_1$ and $x_2$ belong to the same stable manifold, therefore  $\lim \limits_{n \to \infty} d(\varphi^{n}(x), \varphi^{n}(y))=0$ where $d$ is the metric on $S$, and then one has $\ell_{\varphi,f}(x_1)=\ell_{\varphi,f}(x_2)$ for any $f:S\rightarrow \mathbb{R}$ differentiable. Thus, given $x\in K^u$ we can define its {\it Lagrange value} as
$$k_{\varphi,\Lambda,f}(x):=\ell_{\varphi,f}(\tilde{x})=\limsup\limits_{n\to \infty}f(\varphi^n(\tilde{x}))$$
where $\tilde{x}\in \Lambda$ is arbitrary such that $\pi^u(\tilde{x})=x.$ It follows from the definition that $k_{\varphi,\Lambda,f}(K^u)=\mathcal{L}_{\varphi,\Lambda,f}.$ 

In this article, we are mainly interested in the study of the function $D_{\varphi,\Lambda,f}$ defined by
\begin{eqnarray*}
  D_{\varphi,\Lambda,f}:\mathbb{R} &\rightarrow& \mathbb{R} \\
   t &\mapsto& HD(k^{-1}_{\varphi,\Lambda,f}(-\infty,t)).
\end{eqnarray*}
and its relation with the functions given for $t\in \mathbb{R}$ by
\begin{enumerate}
    \item $L_{\varphi,\Lambda,f}(t)=HD(\mathcal{L}_{\varphi,\Lambda, f}\cap(-\infty, t)),$
    \item $M_{\varphi,\Lambda,f}(t)=HD(\mathcal{M}_{\varphi,\Lambda, f}\cap(-\infty, t)),$
    \item $L^{\text{loc}}_{\varphi,\Lambda,f}(t)=\lim\limits_{\epsilon\rightarrow0^+}HD(\mathcal{L}_{\varphi,\Lambda,f}\cap (t-\epsilon,t+\epsilon)).$ 
\end{enumerate}
The study of the continuity of the mentioned functions is closely related to the study of the behavior of the family of sets $\{\Lambda_t\}_{t\in \mathbb{R}}$, where for $t\in \mathbb{R}$
$$\Lambda_t=m_{\varphi, f}^{-1}((-\infinity,t])=\bigcap\limits_{n\in\mathbb{Z}}\varphi^{-n}(f|_{\Lambda}^{-1}((-\infinity,t])) = \{x\in\Lambda: \forall n\in \mathbb{Z}, \ f(\varphi^n(x))\leq t\}.$$
Define then, $R_{\varphi,\Lambda,f}(t)=HD(\Lambda_t).$

It turns out that dynamical Markov and Lagrange spectra associated to hyperbolic dynamics are closely related to the classical Markov and Lagrange spectra. Several results on the Markov and Lagrange dynamical spectra associated to horseshoes in dimension 2 which are analogous to previously known results on the classical spectra were obtained recently. We refer the reader to the book \cite{LMMR} for more information.

In our present work, it is important to mention that in \cite{GCD}, in the context of {\it conservative} diffeomorphism it is proven (as a generalization of the results in \cite{CMM16}) that for typical choices of the dynamic and of the function, the intersections of the corresponding dynamical Markov and Lagrange spectra with half-lines $(-\infty,t)$ have the same Hausdorff dimensions, and this defines a continuous function of $t$ whose image is $[0,\min \{1,\tau\}]$, where $\tau$ is the Hausdorff dimension of the horseshoe. 

Finally, in \cite{LM2} is showed that, for any $n\ge 2$ with $n\neq 3$, the initial segments of the classical spectra until $\sqrt{n^2+4n}$ (i.e., the intersection of the spectra with $(-\infty,\sqrt{n^2+4n}]$) are dynamical Markov and Lagrange spectra associated to a horseshoe $\Lambda(n)$ of some smooth conservative diffeomorphism $\varphi_n$ of $\mathcal{S}^2$ and to some smooth real function $f_n$. Also, in the same article is showed that they are naturally associated to continued fractions with coefficients bounded by $n$. Using this, in \cite{GC} it is proven that for any $t$ that belongs to the closure of the interior of the classical Markov and Lagrange spectra $D(t)=HD(k^{-1}(t))$ and that $D$ is strictly increasing when is restricted to the interior of the spectra. 

Here, in the dynamical setting (with applications to the classical one), we try to characterize the set of points $t\in \mathcal{L}_{\varphi,\Lambda,f}$ that verify $D_{\varphi,\Lambda,f}(t)=HD(k_{\varphi,\Lambda,f}^{-1}(t))$ and describe the properties of $D_{\varphi,\Lambda,f}$ when it is restricted to this set. Some results related with the function $L^{\text{loc}}_{\varphi,\Lambda,f}$ are also given.

\subsection{Statement of the main theorems} 

We write $\text{Diff}^{2}_{\omega}(S)$ for the set of conservative diffeomorphisms of $S$ with respect to a volume form $\omega$. Let $\varphi\in \text{Diff}^{2}_{\omega}(S)$ with a mixing horseshoe $\Lambda$ and for $x\in \Lambda$, let $e^s_x$ and $e^u_x$ unit vectors in the stable and unstable directions of $T_xS$. Given $r\geq 2$, set
$$\mathcal{R}^r_{\varphi, \Lambda}=\{f\in C^r(S,\mathbb{R}): \grad f(x) \textrm{ is not perpendicular neither to } e^s_x \textrm{  nor } e^u_x \textrm{ for all } x\in\Lambda\}$$ 
in other terms, $\mathcal{R}^r_{\varphi, \Lambda}$ is the open set of $C^r$-functions $f:S\to\mathbb{R}$ that are locally monotone along stable and unstable directions. 

Using the notations of the previous subsection, our main theorems are the following
 \begin{theorem}\label{T1}
   Let $\varphi\in \text{Diff}^{2}_{\omega}(S)$ with a mixing horseshoe $\Lambda$. For any $r\ge 2$ and $f\in \mathcal{R}^r_{\varphi,\Lambda}$, the function $D_{\varphi,\Lambda,f}$ is continuous with $D_{\varphi,\Lambda,f}(\mathbb{R})=[0,\frac{HD(\Lambda)}{2}]$   
   and one has a decomposition 
   $$\{t\in \mathcal{L}_{\varphi,\Lambda,f}:D_{\varphi,\Lambda,f}(t)>0\}=\mathcal{J}_{\varphi,\Lambda,f}\cup \mathcal{F}_{\varphi,\Lambda,f}\cup \tilde{\mathcal{J}}_{\varphi,\Lambda,f}$$ 
that satisfies 
\begin{itemize}
    \item $\mathcal{J}_{\varphi,\Lambda,f}\subset \mathcal{L}^{'}_{\varphi,\Lambda,f}$
    \item $D_{\varphi,\Lambda,f}|_{\mathcal{J}_{\varphi,\Lambda,f}}$ is strictly increasing,
    \item $D_{\varphi,\Lambda,f}(\mathcal{J}_{\varphi,\Lambda,f})=(0,\frac{HD(\Lambda)}{2}],$
    \item $D_{\varphi,\Lambda,f}(t)=HD(k_{\varphi,\Lambda,f}^{-1}(t))$ for every $t\in \mathcal{J}_{\varphi,\Lambda,f}$,
    \item $\mathcal{F}_{\varphi,\Lambda,f}$ is countable,
    \item $D_{\varphi,\Lambda,f}(t)\neq HD(k_{\varphi,\Lambda,f}^{-1}(t))$ for every $t\in \tilde{\mathcal{J}}_{\varphi,\Lambda,f}$.
\end{itemize}
Indeed, given $\eta\in(0,\frac{1}{2}HD(\Lambda)]$, define $\eta^-:=\min\{t\in \mathbb{R}:D_{\varphi,\Lambda,f}(t)=\eta\}$. Then, we can set
$$\mathcal{J}_{\varphi,\Lambda,f}=\{\eta^-:\eta\in(0,\frac{1}{2}HD(\Lambda)]\}.$$
\end{theorem}


Now, fix $\varphi_0\in \text{Diff}^{2}_{\omega}(S)$ with a mixing horseshoe $\Lambda_0$ and let $\mathcal{U}$ a $C^{2}$-neighbourhood of $\varphi_0$ in $\text{Diff}^{2}_{\omega}(S)$ such that $\Lambda_0$ admits a continuation $\Lambda (= \Lambda(\varphi))$ for every $\varphi \in \mathcal{U}$.
\begin{remark}
    It is possible to show (see proposition $3.7$ of \cite{C1}) that if $\mathcal{U}\subset\textrm{Diff}^{2}(S)$ is sufficiently small, then there exists a residual subset $\mathcal{U}^{*}\subset \mathcal{U}$ with the property that for every $\varphi\in\mathcal{U}^{*}$ and any $r\geq2$, there exists a $C^r$-residual set $\mathcal{S}^r_{\varphi,\Lambda}\subset \mathcal{R}^r_{\varphi,\Lambda}$ such that given $f\in \mathcal{S}^r_{\varphi,\Lambda}$ one has
    $$\{t\in \mathcal{L}_{\varphi,\Lambda,f}:D_{\varphi,\Lambda,f}(t)>0\}=\mathcal{L}_{\varphi,\Lambda,f}\cap (c_{\varphi,\Lambda,f},\infinity)$$
    where $c_{\varphi,\Lambda,f}=\min \mathcal{L}^{'} _{\varphi,\Lambda,f}$. 
\end{remark}

On the other hand, if the mixing horseshoe $\Lambda_0$ satisfies that $HD(\Lambda_0)<1$ and $\mathcal{U}$ is sufficiently small such that $HD(\Lambda)<1$ for every $\varphi \in \mathcal{U}$, then we have 

\begin{theorem}\label{T2}
The set $\mathcal{R}^r_{\varphi,\Lambda}$ is $C^r$-open and dense and there exists a residual set $\tilde{\mathcal{U}}\subset \mathcal{U}$ such that for every $\varphi\in \tilde{\mathcal{U}}$ and $f\in\mathcal{R}^r_{\varphi,\Lambda}$, one has 
\begin{itemize}
    \item $L^{\text{loc}}_{\varphi,\Lambda,f}(t)=L_{\varphi,\Lambda,f}(t)$ for every $t\in \mathcal{J}_{\varphi,\Lambda,f}$,
    \item $L^{\text{loc}}_{\varphi,\Lambda,f}(t)<L_{\varphi,\Lambda,f}(t)$ for every $t\in \tilde{\mathcal{J}}_{\varphi,\Lambda,f}$.
\end{itemize}
\end{theorem}


Finally, in section \ref{classical}, we use the dynamical characterization of the classical Lagrange spectrum to recover versions of theorems \ref{T1} and \ref{T2} in that context. 
\section{Preliminares}
\subsection{Dynamical defined Cantor sets}

Let $k\ge 1$ be an integer and $\alpha\in [0,1)$ be a real number. A set $K\subset \mathbb{R}$ is called a $C^{k+\alpha}$-{\it regular Cantor set} if there exists a collection $\mathcal{P}=\{I_1,I_2,...,I_r\}$ of compacts intervals and a $C^{k+\alpha}$-expanding map $\psi$, defined in a neighbourhood of $\bigcup \limits_{1\leq j\leq r}I_j$ such that
	
\begin{enumerate}
\item $K\subset \bigcup \limits_{1\leq j\leq r}I_j$ and $\bigcup \limits_{1\leq j\leq r}\partial I_j\subset K$,
		
\item For every $1\leq j\leq r$ we have that $\psi(I_j)$ is the convex hull of a union of $I_r$'s, for $l$ sufficiently large $\psi^l(K\cap I_j)=K$ and $$K=\bigcap_{n\geq 0}\psi^{-n}(\bigcup_{1\leq j\leq r}I_j).$$
\end{enumerate}

Even more, we say that $K$ is {\it non-essentially affine} if there is no global conjugation $h\circ\psi\circ h^{-1}$ such that all branches 
$$(h\circ\psi\circ h^{-1})|_{h(I_j)}, \ j=1, \dots,r$$
are affine maps of the real line.

Given a finite alphabet $B=\{\beta_1,\dots,\beta_r \}$ with $r\geq 2$ of finite sequences of natural numbers such that $\beta_i$ does not begin with $\beta_j$ for $i\neq j$, the Gauss-Cantor $K(B)\subset [0,1]$ associated to $B$ is 
$$K(B)=\{[0;\gamma_1,\gamma_2,\dots]: \gamma_i\in B,\ \ \forall i\in \mathbb{N} \},$$
which is dynamically defined with $I_j=I(\beta_j)=\{[0;\beta_j,a_1,a_2,\dots]:a_i\in \mathbb{N},\ \forall i\in \mathbb{N}\}$ and $\psi|_{I_j}=G^{\abs{\beta_j}}$ where $G$ is the usual Gauss map given by
\begin{eqnarray*}
  G:(0,1) &\rightarrow& [0,1) \\
   x &\mapsto& \dfrac{1}{x}-\left\lfloor \dfrac{1}{x}\right\rfloor.
\end{eqnarray*}
It follows that given $N\geq 2$ the set $C_N=\{x=[0;a_1,a_2,...]: a_i\le N, \forall i\in \mathbb{N}\}$ is a regular Cantor set. Similarly, we conclude the same for the set $\tilde{C}_N=\{1,2,...,N\}+C_N$.

It is proved in \cite{M3} that Gauss-Cantor sets are non-essentially affine. Finally, the following theorem is from \cite{M50}

\begin{theorem}\label{formula}
Let $K_1,K_2$ be regular Cantor sets such that $K_1$ is of class $C^2$ and is non-essentially affine. Then, given a $C^1$ map $g$ from a neighborhood of $K_1\times K_2$ to $\mathbb{R}$ such that in some point of $K_1\times K_2$ its gradient is not parallel to any of the two coordinate axis, we have 
$$HD(g(K_1\times K_2))=\min\{1,HD(K_1\times K_2)\}.$$
\end{theorem}

\subsection{Unstable dimension}
Given a Markov partition $\mathcal{P}=\{R_a\}_{a\in \mathcal{A}}$, recall that the geometrical description of $\Lambda$ in terms of the Markov partition $\mathcal{P}$ has a combinatorial counterpart in terms of the Markov shift $\Sigma_{\mathcal{B}}\subset \mathcal{A}^{\mathbb{Z}}$.
Given an admissible finite sequence $\alpha=(a_0,a_1,...,a_n)\in \mathcal{A}^{n+1}$ (i.e., a factor of some sequence in $\Sigma_{\mathcal{B}}$), we define
	$$I^u(\alpha)=\{x\in K^u: \psi_u^i(x)\in I^u(a_i),\ i=0,1,...,n\}.$$
In a similar way, let $\theta=(a_{s_1},a_{s_1+1},...,a_{s_2})\in \mathcal{A}^{s_2-s_1+1}$ an admissible word where $s_1, s_2 \in \mathbb{Z}$, $s_1 < s_2$ and fix $s_1\le s\le s_2$. Define $$R(\theta;s)=\bigcap_{m=s_1-s}^{s_2-s} \varphi^{-m}(R_{a_{m+s}}).$$
Note that if $x\in R(\theta;s)\cap \Lambda$ then the symbolic representation of $x$ is in the way $\Pi(x)=(\dots,a_{s_1}\dots a_{s-1};a_{s},a_{s+1}\dots a_{s_2}\dots)$, where the letter following to $;$ is in the $0$ position of the sequence.

In our context of dynamically defined Cantor sets, we can relate the length of the unstable intervals determined by an admissible word to its length as a word in the alphabet $\mathcal{A}$ via the {\it bounded distortion property} that let us conclude that for some constant $c_1>0$, and admissible words $\alpha$ and $\beta$
\begin{equation}\label{bdp1}
e^{-c_1}|I^u(\alpha)|\cdot|I^u(\beta)|\le |I^u(\alpha\beta)|\le e^{c_1}|I^u(\alpha)|\cdot|I^u(\beta)|,
	\end{equation}
if $w_1,w_2\in I^u(\alpha)$ but $w_1$ and $w_2$ do not belong to the same $I^u(\alpha b)$ for any $b\in\mathcal{A}$, then
\begin{equation}\label{bdp2}
\abs{w_1-w_2}\geq e^{-c_1}\abs{I^u(\alpha)}
	\end{equation}
and also that for some positive constants $\lambda_1,\lambda_2<1$, one has
\begin{equation}\label{bdp3}
e^{-c_1} \lambda_1^{\abs{\alpha}}\leq\abs{I^u(\alpha)}\leq e^{c_1} \lambda_2^{\abs{\alpha}}.
	\end{equation}

Given $\alpha=(a_0,a_1,...,a_n)\in \mathcal{A}^{n+1}$ admissible write, $r^{u}(\alpha)=\lfloor \log(1/\abs{I^u(\alpha)})\rfloor$ 
and for $r\in \mathbb{N}$ define the set
$$P^{u}_r=\{(a_0,a_1,...,a_n)\in \mathcal{A}^{n+1} \ \mbox{admissible}:  r^{u}((a_0,...,a_n))\geq r \ \mbox{and} \ r^{u}((a_0,...,a_{n-1}))<r\}$$
Given $X\subset \Lambda$ compact and $\varphi$-invariant, we also set
$$\mathcal{C}_u(X,r)=\{\alpha\in P^u_r:  I^{u}(\alpha)\cap \pi^u(X)\neq \emptyset\}.$$
The following result is from \cite{C1}:
\begin{lemma}
For each $X\subset \Lambda$ compact and $\varphi$-invariant, the limit capacity of $\pi^u(X)$ is given by the limit 
$$D_u(X):=\lim_{r\to \infty}\dfrac{\log |\mathcal{C}_u(X,r)|}{r}.$$
\end{lemma}
\begin{remark}
    A similar result can be established for $\pi^s(X)$.
\end{remark}

\subsection{Results on dynamical spectra}
Let $r\geq2$ and define
$$\mathcal{P}^r_{\varphi,\Lambda}=\{f\in C^r(S,\mathbb{R}): \nabla f(x)\neq 0,  \ \forall \ x\in \Lambda\}.$$ 
In other words, $\mathcal{P}^r_{\varphi,\Lambda}$ is the class of functions $C^r$, $f :S \rightarrow \mathbb{R}$ such that for every $x\in \Lambda$ either $\grad f(x)$ is not perpendicular to $e^s_x$ or is not perpendicular to $e^u_x$. 

In \cite{CMM16} and \cite{GCD} are proven the following results

\begin{theorem}\label{c1}
Let $\varphi\in \text{Diff}^{2}_{\omega}(S)$ with a mixing horseshoe $\Lambda$. For every $r\ge 2$ the set $\mathcal{P}^r_{\varphi,\Lambda}$ is $C^r$-open and dense and such that for any $f\in \mathcal{P}^r_{\varphi,\Lambda}$ the function
$$t\rightarrow D_u(\Lambda_t) $$
is continuous and we have the equality
$$R_{\varphi,\Lambda,f}(t)=2D_u(\Lambda_t).$$
Even more,  $\mathcal{R}^r_{\varphi,\Lambda}$ is also $C^r$-open and dense if $HD(\Lambda)<1$.
\end{theorem}

\begin{theorem}\label{c2}

Let $\varphi_0\in \text{Diff}^{2}_{\omega}(S)$ with a mixing horseshoe $\Lambda_0$ with $HD(\Lambda_0)<1$ and $\mathcal{U}$ a $C^{2}$-sufficiently small neighbourhood of $\varphi_0$ in $\text{Diff}^{2}_{\omega}(S)$ such that $\Lambda_0$ admits a continuation $\Lambda$ with $HD(\Lambda)<1$ for every $\varphi \in \mathcal{U}$. There exists a residual set $\tilde{\mathcal{U}}\subset \mathcal{U}$ such that for every $\varphi\in \tilde{\mathcal{U}}$ and $f\in\mathcal{R}^r_{\varphi,\Lambda}$, we have the equality 
$$R_{\varphi,\Lambda,f}=L_{\varphi,\Lambda,f}=M_{\varphi,\Lambda,f}.$$
\end{theorem}
\subsection{Sets of finite type and connection of subhorseshoes}\label{tipofinito}

The following definitions and results can be found in \cite{GC}. Fix a horseshoe $\Lambda$ of some conservative diffeomorphism $\varphi:S\rightarrow S$ and $\mathcal{P}=\{R_a\}_{a\in \mathcal{A}}$ some Markov partition for $\Lambda$. Take a finite collection $X$ of finite admissible words $\theta=(a_{-n(\theta)},\dots,a_{-1},a_0,a_1,\dots,a_{n(\theta)})$, we said that the maximal invariant set 
$$M(X)=\bigcap \limits_{m \in \mathbb{Z}} \varphi ^{-m}(\bigcup \limits_{\theta \in X}  R(\theta;0))$$ 
is a \textit{hyperbolic set of finite type}. Even more, it is said to be a \textit{subhorseshoe} of $\Lambda$ if it is nonempty and $\varphi|_{M(X)}$ is transitive. Observe that a subhorseshoe need not be a horseshoe; indeed, it could be a periodic orbit in which case it will be called of trivial.

By definition, hyperbolic sets of finite type have local product structure. In fact, any hyperbolic set of finite type is a locally maximal invariant set of a neighborhood of a finite number of elements of some Markov partition of $\Lambda$.

\begin{definition}
Any $\tau \subset M(X)$ for which there are two different subhorseshoes $\Lambda(1)$ and $\Lambda(2)$ of $\Lambda$ contained in $M(X)$ with 
$$\tau=\{x\in M(X):\ \omega(x)\subset \Lambda(1)\ \text{and}\ \alpha(x)\subset \Lambda(2)  \}$$
will be called a transient set or transient component of $M(X)$.
\end{definition}

Note that by the local product structure, given a transient set $\tau$ as before,
\begin{equation}\label{transient}
    HD(\tau)=HD(K^s(\Lambda(2)))+HD(K^u(\Lambda(1)))
\end{equation}
and also, for any subhorseshoe $\tilde{\Lambda}\subset \Lambda$, being  $\varphi$ conservative, one has
\begin{equation}\label{transient2}
HD(\tilde{\Lambda})=HD(K^s(\tilde{\Lambda}))+HD(K^u(\tilde{\Lambda}))=2HD(K^u(\tilde{\Lambda})). 
\end{equation}
\begin{proposition}\label{appendix}
Any hyperbolic set of finite type $M(X)$, associated with a finite collection of finite admissible words $X$ as before, can be written as
$$M(X)=\bigcup \limits_{i\in \mathcal{I}} \tilde{\Lambda}_i $$ 
where $\mathcal{I}$ is a finite index set (that may be empty) and for $i\in \mathcal{I}$,\ $\tilde{\Lambda}_i$ is a subhorseshoe or a transient set.
\end{proposition}

Fix $f:S\rightarrow \mathbb{R}$  differentiable. A notion that plays an important role in our study of the concentration of Hausdorff dimension is the notion of \textit{connection of subhorseshoes} 

\begin{definition}\label{conection of horseshoes1}
Given $\Lambda(1)$ and $\Lambda(2)$ subhorseshoes of $\Lambda$ and $t\in \mathbb{R}$, we said that $\Lambda(1)$ \emph{connects} with $\Lambda(2)$ or that $\Lambda(1)$ and $\Lambda(2)$ \emph{connect} before $t$ if there exist a subhorseshoe $\tilde{\Lambda}\subset \Lambda$ and some $q< t$ with $\Lambda(1) \cup \Lambda(2) \subset \tilde{\Lambda}\subset \Lambda_q$.
\end{definition}

For our present purposes, the next criterion of connection will be also important

\begin{proposition}\label{connection11}
Suppose $\Lambda(1)$ and $\Lambda(2)$ are subhorseshoes of $\Lambda$ and for some $x,y \in \Lambda$ we have $x\in W^u(\Lambda(1))\cap W^s(\Lambda(2))$ and $y\in W^u(\Lambda(2))\cap W^s(\Lambda(1))$. If for some $t\in \mathbb{R}$, it is true that 
$$\Lambda(1) \cup \Lambda(2) \cup \mathcal{O}(x) \cup \mathcal{O}(y) \subset \Lambda_t,$$ then for every $\epsilon >0$,\ $\Lambda(1)$ and $\Lambda(2)$ connect before $t+\epsilon$. 
\end{proposition}

 \begin{corollary}\label{connection3}
 Let $\Lambda(1)$,\ $\Lambda(2)$ and $\Lambda(3)$ subhorseshoes of $\Lambda$ and $t\in \mathbb{R}$. If $\Lambda(1)$\ connects with $\Lambda(2)$ before $t$ and $\Lambda(2)$\ connects with $\Lambda(3)$ before $t$. Then also $\Lambda(1)$\ connects with $\Lambda(3)$ before $t$.
 \end{corollary}

\section{Proof of the main theorems}
For Theorem \ref{T1}, given $\eta\in(0,\frac{1}{2}HD(\Lambda)]$, we show that  $\eta^-\in \mathcal{L}_{\varphi,\Lambda,f}$ satisfies $\eta=D_{\varphi,\Lambda,f}(\eta^-)=HD(k_{\varphi,\Lambda,f}^{-1}(\eta^-))$. Then we prove that (modulus some countable subset), these are all the points $t\in \mathcal{L}_{\varphi,\Lambda,f}$ that satisfy $D_{\varphi,\Lambda,f}(t)=HD(k_{\varphi,\Lambda,f}^{-1}(t))>0$. For Theorem \ref{T2}, we use the previous theorem and that generically one has the equality $L_{\varphi,\Lambda,f}=2D_{\varphi,\Lambda,f}$.

\subsection{Relation between dimensions}
Fix $\varphi\in \text{Diff}^{2}_{\omega}(S)$ with a mixing horseshoe $\Lambda$. We start the proof by relating the functions $D_{\varphi,\Lambda,f}$ and $R_{\varphi,\Lambda,f}$ for $f\in\mathcal{P}^r_{\varphi,\Lambda}$
\begin{proposition}\label{R}
Given $f\in\mathcal{P}^r_{\varphi,\Lambda}$ and $t\in \mathbb{R}$, one has
$$D_{\varphi,\Lambda,f}(t)=HD(k^{-1}_{\varphi,\Lambda,f}(-\infty,t])=\frac{1}{2}R_{\varphi,\Lambda,f}(t).$$
\end{proposition}

\begin{proof}
Let $x\in \Lambda$ with $\ell_{\varphi,f}(x)=\eta < t$, then there exist a sequence $\{ n_k \}_{k\in\mathbb{N}}$ such that $\lim \limits_{k \to \infinity} f(\varphi^{n_k}(x))=\eta$. By compactness, without loss of generality, we can also suppose that $\lim \limits_{k \to \infinity} \varphi^{n_k}(x)=y$ for some $y\in \Lambda$ and so that $f(y)=\eta$. We affirm that $m_{\varphi, f}(y)=\eta$: in other case we would have for some $\tilde{k} \in \mathbb{Z}$ and $q \in \mathbb{R}$, $f(\varphi^{\tilde{k}}(y))>q>\eta$ and then for $k$ big enough by continuity $f(\varphi^{\tilde{k}+n_k}(x))>\eta$ that contradicts the definition of $\eta$.

Now, consider $\tilde{N}$ big enough such that if for two elements $a, b\in \Lambda$ their kneading sequences coincide in the central block (centered at the zero position) of size $2\tilde{N}+1$ then $\abs{f(a)-f(b)}<(t-\eta)/4$ and $N$ big enough such that for $k\geq N$ one has $f(\varphi^{k}(x))<\eta+(t-\eta)/4$ and the kneading sequences of $\varphi^{n_k}(x)$ and $y$ coincide in the central block of size $2\tilde{N}+1$. Suppose $\Pi(x)= (x_n)_{n\in \mathbb{Z}}$ and $\Pi(y)= (y_n)_{n\in \mathbb{Z}}$, then the point
$$\tilde{y}=\Pi^{-1}(\dots,y_{-n},\dots,y_{-1};x_{n_N},x_{n_N+1},\dots)$$
satisfies $\tilde{y}\in \Lambda_{(t+\eta)/2}$ and $\ell_{\varphi,f}(\tilde{y})=\eta.$ 

Therefore, $w=\pi^u(x)$ satisfies that 
$$\psi^{n_N}_u(w)=\pi^u(\tilde{y})\in \pi^u(\Lambda_{(t+\eta)/2}) \subset \bigcup\limits_{s<t}\pi^u(\Lambda_s)\subset \pi^u(\Lambda_t).$$ 
This implies that $k^{-1}_{\varphi,\Lambda,f}(-\infty,t)\subset \bigcup\limits_{n\in \mathbb{N}}\psi_u^{-n}(\pi^u(\Lambda_t))$ and then $HD(k^{-1}_{\varphi,\Lambda,f}(-\infty,t))\leq HD(\pi^u(\Lambda_t))\leq D_u(\Lambda_t)$. As always is true that $\bigcup\limits_{s<t}\pi^u(\Lambda_s)\subset k^{-1}_{\varphi,\Lambda,f}(-\infty,t)$, because theorem \ref{c1}, one has
\begin{eqnarray*}
\frac{1}{2}R_{\varphi,\Lambda,f}(t)&=&D_u(\Lambda_t)=\sup\limits_{s<t}D_u(\Lambda_s)\leq HD(k^{-1}_{\varphi,\Lambda,f}(-\infty,t))\\ &\leq& HD(k^{-1}_{\varphi,\Lambda,f}(-\infty,t])\leq \inf\limits_{s>t}HD(k^{-1}_{\varphi,\Lambda,f}(-\infty,s))\\&\leq&\inf\limits_{s>t}\frac{1}{2}R_{\varphi,\Lambda,f}(s)=\frac{1}{2}R_{\varphi,\Lambda,f}(t), 
\end{eqnarray*}
as we wanted to see.
\end{proof}
\begin{remark}\label{r1}
From the proof of the proposition we get for $f\in\mathcal{P}^r_{\varphi,\Lambda}$ and any $t\in \mathbb{R}$ that
$$k^{-1}_{\varphi,\Lambda,f}(-\infty,t)\subset \bigcup\limits_{n\in \mathbb{N}}\psi_u^{-n}(\pi^u(\Lambda_t)).$$
\end{remark}
\begin{corollary}
If $f\in\mathcal{P}^r_{\varphi,\Lambda}$ and $t\in \mathbb{R}$ is such that $R_{\varphi,\Lambda,f}(t)=0$, then 
$$D_{\varphi,\Lambda,f}(t)=HD(k_{\varphi,f}^{-1}(t))=0.$$
\end{corollary}
\subsection{Sequence of subhorseshoes}\label{horse}

First, observe that $D_{\varphi,\Lambda,f}(\mathbb{R})=\frac{1}{2}R_{\varphi,\Lambda,f}(\mathbb{R})=[0,\frac{1}{2}HD(\Lambda)]$. As in the statement of Theorem \ref{T1}, let us consider the set 
$$\mathcal{J}_{\varphi,\Lambda,f}=\{\eta^-:\eta\in(0,\frac{1}{2}HD(\Lambda)]\},$$
where $\eta^-=\min\{t\in \mathbb{R}:D_{\varphi,\Lambda,f}(t)=\eta\}$ for $\eta\in(0,\frac{1}{2}HD(\Lambda)]$ . Clearly, $D_{\varphi,\Lambda,f}(\mathcal{J}_{\varphi,\Lambda,f})=(0,\frac{1}{2}HD(\Lambda)]$ and $D_{\varphi,\Lambda,f}|_{\mathcal{J}_{\varphi,\Lambda,f}}$ is strictly increasing.

Fix $\eta$, as before, and $\epsilon>0$ such that $D_{\varphi,\Lambda,f}(\eta^--\epsilon)>0.999D_{\varphi,\Lambda,f}(\eta^-)$. By definition, one can find a strictly increasing sequence $\{t_n\}_{n\geq 0}$ such that  $\lim\limits_{n\rightarrow\infty}t_n=\eta^-$, $t_0=\eta^--\epsilon$ and $D_{\varphi,\Lambda,f}$ is injective on $\{t_n:n\geq 0\}$. 

Now, proposition $1$ of \cite{GCD} applied to $t_{n+1}$ and $\eta_n=1-R_{\varphi,\Lambda,f}(t_n)/R_{\varphi,\Lambda,f}(t_{n+1})$ let us find $\delta_n>0$ and some subhorseshoe $\Lambda^n\subset \Lambda_{t_{n+1}-\delta_n}$ with
\begin{eqnarray*}
HD(\Lambda^n)=2HD(K^u(\Lambda^n))&>&2(1-\eta_n)D_u(\Lambda_{t_{n+1}})=(1-\eta_n)HD(\Lambda_{t_{n+1}})\\&=&(1-\eta_n)R_{\varphi,\Lambda,f}(t_{n+1})=R_{\varphi,\Lambda,f}(t_n)
\end{eqnarray*}
and in particular $\max f|_{\Lambda^n}>t_n$. Also, for $n\geq 0$
\begin{equation}\label{estima1}
D_u(\Lambda^n)=\frac{1}{2}HD(\Lambda^n)>\frac{1}{2}R_{\varphi,\Lambda,f}(t_n)\geq\frac{1}{2}R_{\varphi,\Lambda,f}(t_0)=D_u(\Lambda_{t_0})>0.999D_u(\Lambda_{\eta^-}).
\end{equation}
Now, take $r_0$ big enough such that $2^{2024}<\abs{\mathcal{C}_u(\Lambda_{\eta^-},r_0)}$ and
\begin{equation}\label{estima3}
  \frac{\log \abs{\mathcal{C}_u(\Lambda_{\eta^-},r_0)}}{r_0-c_1}<1.001 D_u(\Lambda_{\eta^-}). 
\end{equation}
We set $\mathcal{B}_0=\mathcal{C}_u(\Lambda_{\eta^-},r_0)$, $N_0=\abs{\mathcal{C}_u(\Lambda_{\eta^-},r_0)}$ and for $n\geq 0$ and $M\in \mathbb{N}$ define the set
$$\mathcal{B}_M(\Lambda^n)=\{\beta=\beta_1\dots\beta_M :\forall \, 1\leq j\leq M, \,\, \beta_j\in\mathcal{B}_0 \,\,  \textrm{ and } \,\, \pi^u(\Lambda^n)\cap I^u(\beta)\neq\emptyset\}.$$

Before continuing, we introduce some notation. Consider $\beta=\beta_{k_1}\beta_{k_2}...\beta_{k_ \ell}=a_1...a_p \in \mathcal{A}^{p}, \ \beta_{k_i}\in \mathcal{B}_0, \ 1\le i\le \ell$. We say that $n\in \{1,...,p\}$ is the nth position of $\beta$. If $\beta_{k_i}\in \mathcal{A}^{n_{k_{i}}} $ we write $|\beta_{k_i}|=n_{k_i}$ for its length and $P(\beta_{k_i})=\{1,2,...,n_{k_i}\}$ for its set of positions as a word in the alphabet $\mathcal{A}$ and given $s \in P(\beta_{k_i})$ we call  $P(\beta,k_i;s)=n_{k_1}+...+n_{k_{i-1}}+s$ the position in $\beta$ of the position $s$ of $\beta_{k_i}$.

Recall that the sizes of the intervals $I^u(\alpha)$ behave essentially submultiplicatively due the bounded distortion property of $\psi_u$ (equation (\ref{bdp1})) so that, one has   
$$|I^u(\beta)|\leq \exp(-M(r_0-c_1))$$
for any $\beta\in\mathcal{B}_M(\Lambda^n)$, and thus, $\{I^u(\beta):\beta\in \mathcal{B}_M(\Lambda^n)\}$ is a covering of $\pi^u(\Lambda^n)$ by intervals of sizes $\leq \exp(-M(r_0-c_1))$. In particular for $M(\Lambda^n)=M_n$ sufficiently large 

\begin{eqnarray*}
\dfrac{\log\abs{\mathcal{B}_{M_n}(\Lambda^n)}}{\log N_0^{M_n}}&=&\dfrac{\dfrac{\log\abs{\mathcal{B}_{M_n}(\Lambda^n)}}{-\log \exp (-M_n (r_0-c_1))}}{\dfrac{M_n\cdot \log N_0}{M_n(r_0-c_1)}}\\ &\geq& \frac{\dfrac{\log\abs{\mathcal{B}_{M_n}(\Lambda^n)}}{-\log \exp (-M_n (r_0-c_1))}}{1.001 D_u(\Lambda_{\eta^-})} \quad (\textrm{by equation \ref{estima3}}) \\&\geq& \frac{0.999D_u(\Lambda^n)}{1.001 D_u(\Lambda_{\eta^-})} \quad (\textrm{$M_n$ is big}) \\&\geq& \dfrac{0.999\cdot0.999D_u(\Lambda_{\eta^-})}{1.001 D_u(\Lambda_{\eta^-})}\quad (\textrm{by equation \ref{estima1}})\\&\geq& \dfrac{0.999\cdot0.999}{1.001} \\&>& \frac{991}{1000}.
\end{eqnarray*}
Then we have proved the next result 
\begin{lemma}\label{Btam}
Given $n\geq 0$ and $M_n$ large 
$$ \abs{\mathcal{B}_{M_n}(\Lambda^n)}\geq N_0^{\frac{991}{1000}\cdot M_n}.$$
\end{lemma}

Given $n\geq0$ consider $N_n$ large enough such that for two elements $x,y\in \Lambda$ if their kneading sequences coincide in the central block (centered at the zero position) of size $2N_n+1$ then $\abs{f(x)-f(y)}<\delta_n/2$. Now, for $M\in \mathbb{N}$ and $\beta=\beta_1\dots\beta_M \in \mathcal{B}_M(\Lambda^n) $ with $\beta_i\in\mathcal{B}_0$ for all $1\leq i\leq M$, in \cite{C1} is defined the notion of \emph{M-good positions} of $\beta$ for index $j\in\{1,\dots,M\}$ that allows us to have some control over the values that $f$ takes in some rectangles. 

More specifically, in the mentioned paper is defined $k\in \mathbb{N}$ which does not depend on $n$ in such a way that most positions (more than 98\%) of some word $\beta_n \in \mathcal{B}_{5N_n k}(\Lambda^{n})$ are $5N_n k$-good and the next proposition holds 
\begin{proposition}\label{control}
If $\beta_n=\beta^n_1\beta^n_2\dots\beta^n_{5N_nk}$ with $\beta^n_r\in \mathcal{B}_0$ for $i=1, \dots, 5N_nk$ and for some $1<i<j<5N_nk$, the positions $i-1,i,j,j+1$ are $5N_nk$-good positions of $\beta_n$ and $j-i \geq k/(8N_0^2)$. Then for each $i\leq s \leq j$ and $\bar{n} \in P(\beta^n_s)$ if $\zeta=\beta^n_{i-1}\beta^n_{i}\dots\beta^n_j\beta^n_{j+1}$ and $x\in R(\zeta;P(\zeta,s;\bar{n}))\cap \Lambda$, we have $f(x)<t_{n+1}$.
\end{proposition}

Using these results, in \cite{C1} is constructed a function
$$O: \mathbb{N}\cup \{0\} \rightarrow \bigcup\limits_{j=2}^{k-1}\mathcal{B}_0^j,$$
with the property that if for some $m,n\in \mathbb{N}\cup \{0\}$ one has $O(m)=O(n)$ then it is possible to go from $\Lambda^{m}$ to $\Lambda^{n}$ (and also from $\Lambda^{n}$ to $\Lambda^{m}$), without leaving $\Lambda_{\max\{t_{n+1},t_{m+1}\}}$ and staying arbitrarily close to the orbit of the periodic point $p_{n,m}=\Pi^{-1}(\overline{O(n)})=\Pi^{-1}(\overline{O(m)})$ for arbitrarily long times, where $\overline{O(n)}$ is the infinite sequence with period $O(n)$. Then, proposition \ref{connection11} let us conclude

\begin{proposition}\label{contradiction}
 Let $m,n\in \mathbb{N}\cup \{0\}$ such that $O(m)=O(n)$. Then $\Lambda^m$ connects with $\Lambda^n$ before $\max\{t_{n+1},t_{m+1}\}$.
 \end{proposition}

As the function $O$ takes only a finite number of different values, by proposition \ref{contradiction}, without loss of generality, we can suppose that $\Lambda^n$ connects with $\Lambda^m$ before $\max\{t_{n+1},t_{m+1}\}$, for any $n,m\in \mathbb{N}\cup \{0\}$. Corollary \ref{connection3} let us construct inductively a sequence of subhorseshoes $\{\tilde{\Lambda}^n\}_{n\geq0}$ such that for any $n\geq0$ 
\begin{itemize}
    \item $\Lambda^0\cup \dots\cup \Lambda^n\subset \tilde{\Lambda}^n\subset \Lambda_{q_n}$ for some $q_n<t_{n+1}$,
    \item $\tilde{\Lambda}^n\subset \tilde{\Lambda}^{n+1}$.
\end{itemize}
As a consequence 
$$R_{\varphi,\Lambda,f}(t_n)<HD(\Lambda^n)\leq HD(\tilde{\Lambda}^n)\leq R_{\varphi,\Lambda,f}(t_{n+1})$$
and 
$$t_n<\max f|_{\Lambda^n}\leq \max f|_{\tilde{\Lambda}^n}<t_{n+1}$$
we resume our conclusions in the following proposition
\begin{proposition}\label{horseshoes}
Given $\eta\in(0,\frac{1}{2}HD(\Lambda)]$ and $\epsilon>0$ small, one can find a strictly increasing sequence $\{t_n\}_{n\geq 0}$ with $t_0=\eta^--\epsilon$ and a sequence $\{\tilde{\Lambda}^n\}_{n\geq0}$ of subhorseshoes of $\Lambda$ with the following properties 
\begin{itemize}
    \item $\lim\limits_{n\rightarrow\infty}t_n=\eta^-$,
    \item $\tilde{\Lambda}^n\subset \tilde{\Lambda}^{n+1}$,
    \item $t_n<\max f|_{\tilde{\Lambda}^n}<t_{n+1}$,
    \item $R_{\varphi,\Lambda,f}(t_n)<HD(\tilde{\Lambda}^n)\leq R_{\varphi,\Lambda,f}(t_{n+1})$.
\end{itemize}
\end{proposition}

\begin{remark}
Note that unless passing to the sequence $\{t_{2n}\}_{n\in\mathbb{N}}$ we can suppose that $R_{\varphi,\Lambda,f}(t_n)<HD(\tilde{\Lambda}^n)< R_{\varphi,\Lambda,f}(t_{n+1})$.
\end{remark}

\subsection{Putting unstable Cantor sets into $k^{-1}_{\varphi,\Lambda,f}(\eta^-)$}

Now, we will construct a homeomorphism $\Theta:K^u(\tilde{\Lambda}^0)\rightarrow k^{-1}_{\varphi,\Lambda,f}(\eta^-)$ with H\"older inverse with exponent arbitrarily close to one. By the spectral theorem, we can suppose without loss of generality that the subhorseshoes $\tilde{\Lambda}^n$ of proposition \ref{horseshoes} are mixing.

The argument is similar to that of \cite{GC}, we write all the details here for completeness. Given $n\geq 0$, there exits $c(n)\in \mathbb{N}$ such that given two letters $a$ and $\bar{a}$ in the alphabet $\mathcal{A}(\tilde{\Lambda}^n)$ of $\tilde{\Lambda}^n$ one can find some word $(a_1,\dots ,a_{c(n)})$ in the same alphabet such that $(a,a_1,\dots ,a_{c(n)},\bar{a})$ is admissible. Given $a$ and $\bar{a}$ we will consider always a fixed $(a_1,\dots ,a_{c(n)})$ as before. 

Now, given $n\geq 1$, consider the kneading sequence $\{x^n_r\}_{r\in \mathbb{Z}}$ of some point $x_n\in \tilde{\Lambda}^n$ such that $f(x_n)=\max f|_{\tilde{\Lambda}^n}$. As $\tilde{\Lambda}^n$ is a subhorseshoe of $\Lambda$, it is the invariant set in some rectangles determined for a set of words of size $2p(n)+1$ for some $p(n)\in \mathbb{N}$. Take then $r(n)>p(n+1)+p(n)+p(n-1)$ big enough such that for any $\alpha=(a_0, a_{1} \cdots, a_{2r(n)})\in {\mathcal{A}}^{2r(n)+1}$ and $z,y\in R(\alpha;r(n))$ one has $\abs{f(x)-f(y)}<\min \{(t_{n+1}-f(x_n))/2,(f(x_n)-t_n)/2\}.$ Finally, define 
$$s(n)=\sum_{k=1}^{n}(2r(k)+2c(k)+1).$$

Given $a\in K^u(\tilde{\Lambda}^0)$ with kneading sequence $(a_0,a_1,a_2,\dots )$ (write $a\sim (a_0,a_1,a_2,\dots )$) for $n\geq 1$ set $a^{(n)}=(a_{s(n)!+1},\dots, a_{s(n+1)!})$, so one has
$$a\sim (a_0,a_1,a_2,\dots )=(a_0,a_1,\dots,a_{s(1)!},a^{(1)}, a^{(2)},\dots ,a^{(n)}, \dots).$$ 
Define then $\Theta(a)\in K^u$ by
$$\Theta(a)\sim (a_0,a_1,\dots,a_{s(1)!},h_{1}, a^{(1)},h_{2},a^{(2)},\dots ,h_{n},a^{(n)},h_{n+1}, \dots)$$
where
$$h_n=(c_1^n, x^n_{-r(n)},\dots ,x^n_{-1},x^n_0,x^n_1,\dots,x^n_{r(n)},c_2^n)$$
and $c_1^n$ and $c_2^n$ are words in the original alphabet $\mathcal{A}$ of $\Lambda$ with $\abs{c_1^n}=\abs{c_2^n}=c(n)$ such that $(a_0,a_1,\dots,a_{s(1)!},h_{1}, a^{(1)},h_{2},\dots ,h_{n},a^{(n)})$ appears in the kneading sequence of some point of $\Lambda^n$.

It follows from the construction of $\Theta$ that $k_{\varphi,\Lambda,f}(\Theta(a))=\eta^-$ for every $a\in K^u(\tilde{\Lambda}^0)$, so we have defined the map
\begin{eqnarray*}
  \Theta:K^u(\tilde{\Lambda}^0) &\rightarrow& k^{-1}_{\varphi,\Lambda,f}(\eta^-) \\
   a &\rightarrow& \Theta(a)
\end{eqnarray*}
that is clearly continuous and injective.

On the other hand, if $\tilde{a}_1$ and $\tilde{a}_2$ are such that their kneading sequences are equal up to the $s$-nth letter and $n \in \mathbb{N}$ is maximal such that $s(n)!<s$ then, because $\abs{h_k}=2r(k)+2c(k)+1$, $\Theta(\tilde{a}_1)$ and $\Theta(\tilde{a}_2)$ coincide exactly in their first 
$$s+\sum_{k=1}^{n}2r(k)+2c(k)+1=s+s(n)$$
letters.

So, given $\rho>0$ small, if $s$ is big such that $s(n)/(s+s(n))<\frac{\rho \log \lambda_2}{\log \lambda_1-4c_1}$, using equations \ref{bdp1}, \ref{bdp2} and \ref{bdp3}, we have 

\begin{eqnarray*}
 && \abs{\Theta(\tilde{a}_1)-\Theta(\tilde{a}_2)}^{1-\rho}\\ &\geq& e^{-(1-\rho)c_1}\cdot\abs{I^u(a_1,\dots,a_{s(1)!},h_{1}, a^{(1)},\dots ,a^{(n-1)},h_{n},a_{s(n)!+1},\dots, a_s)}^{1-\rho}\\
  &=& e^{-(1-\rho)c_1}\cdot\abs{I^u(a_1,\dots,a_{s(1)!},h_{1}, a^{(1)},\dots ,a^{(n-1)},h_{n},a_{s(n)!+1},\dots, a_s)}\cdot\\ && \abs{I^u(a_1,\dots,a_{s(1)!},h_{1}, a^{(1)},\dots ,a^{(n-1)},h_{n},a_{s(n)!+1},\dots, a_s)}^{-\rho}\\ &\geq& \frac{e^{-(1-\rho)c_1}}{e^{2nc_1}}\cdot\abs{I^u(a_1,\dots,a_{s(1)!})}\cdot\abs{I^u(a^{(1)})} \dots  
 \abs{I^u(a^{(n-1)})}\cdot \\ &&\abs{I^u(a_{s(n)!+1},\dots, a_s)}\cdot\abs{I^u(h_1)}\dots \abs{I^u(h_n)}\cdot \\&& \abs{I^u(a_1,\dots,a_{s(1)!},h_{1}, a^{(1)},\dots ,a^{(n-1)},h_{n},a_{s(n)!+1},\dots, a_s)}^{-\rho} \\ &\geq&\frac{e^{-(1-\rho)c_1}}{e^{3nc_1}}\cdot\abs{I^u(a_1,a_2,\dots,a_s)}\cdot\abs{I^u(h_1)}\dots \abs{I^u(h_n)}\cdot 
 \\ &&\abs{I^u(a_1,\dots,a_{s(1)!},h_{1}, a^{(1)},\dots a^{(n-1)},h_{n},a_{s(n)!+1},\dots, a_s)}^{-\rho}\\ 
 &\geq&e^{-(1-\rho)c_1}\cdot\abs{I^u(a_1,a_2,\dots,a_s)}\cdot e^{-4nc_1}e^{s(n)\cdot\log \lambda_1}\cdot \\ &&\abs{I^u(a_1,\dots,a_{s(1)!},h_{1}, a^{(1)},\dots a^{(n-1)},h_{n},a_{s(n)!+1},\dots, a_s)}^{-\rho}\\ &\geq&e^{-(1-\rho)c_1}\cdot\abs{I^u(a_1,a_2,\dots,a_s)}\cdot e^{(\log \lambda_1-4c_1)s(n)}\cdot\\ &&\abs{I^u(a_1,\dots,a_{s(1)!},h_{1}, a^{(1)},\dots a^{(n-1)},h_{n},a_{s(n)!+1},\dots, a_s)}^{-\rho}\\ &\geq&e^{-(1-\rho)c_1}\cdot\abs{I^u(a_1,a_2,\dots,a_s)}\cdot e^{\rho  (s+s(n))\log \lambda_2}\cdot\\ &&\abs{I^u(a_1,\dots,a_{s(1)!},h_{1}, a^{(1)},\dots a^{(n-1)},h_{n},a_{s(n)!+1},\dots, a_s)}^{-\rho}\\  
\end{eqnarray*}
  
\begin{eqnarray*}
  &\geq&\frac{e^{-(1-\rho)c_1}}{e^{\rho c_1}}\cdot\abs{I^u(a_1,a_2,\dots,a_s)}\cdot\\
  &&\abs{I^u(a_1,\dots,a_{s(1)!},h_{1}, a^{(1)},\dots ,a^{(n-1)},h_{n},a_{s(n)!+1},\dots, a_s)}^{\rho}\cdot\\ &&\abs{I^u(a_1,\dots,a_{s(1)!},h_{1}, a^{(1)},\dots a^{(n-1)},h_{n},a_{s(n)!+1},\dots, a_s)}^{-\rho}\\ 
  &\geq&e^{-c_1}\cdot \abs{\tilde{a}_1-\tilde{a}_2}.
\end{eqnarray*}

Therefore the map $\Theta^{-1}:\Theta(K^u(\tilde{\Lambda}^0)) \rightarrow K^u(\tilde{\Lambda}^0)$ is a H\"older map with exponent $1-\rho$ and then
\begin{eqnarray*}
   HD(K^u(\tilde{\Lambda}^0))=HD(\Theta^{-1}(\Theta(K^u(\tilde{\Lambda}^0))))&\leq& \frac{1}{1-\rho}\cdot HD(\Theta(K^u(\tilde{\Lambda}^0)))\\ &\leq& \frac{1}{1-\rho}\cdot HD(k^{-1}_{\varphi,\Lambda,f}(\eta^-)).
\end{eqnarray*}
Letting $\rho$ go to zero, we obtain
$$HD(K^u(\tilde{\Lambda}^0))\leq HD(k^{-1}_{\varphi,\Lambda,f}(\eta^-)).$$
Therefore, by proposition \ref{horseshoes}
$$D_{\varphi,\Lambda,f}(\eta^--\epsilon)=\frac{1}{2}R_{\varphi,\Lambda,f}(t_0) < \frac{1}{2}HD(\tilde{\Lambda}^0)= HD(K^u(\tilde{\Lambda}^0))\leq HD(k^{-1}_{\varphi,\Lambda,f}(\eta^-)).$$
Letting $\epsilon$ tend to zero, by proposition \ref{R} we have 
$$HD(k_{\varphi,\Lambda,f}^{-1}(-\infty,\eta^-])=D_{\varphi,\Lambda,f}(\eta^-)\leq HD(k^{-1}_{\varphi,\Lambda,f}(\eta^-))$$
and as the other inequality always is true
$$\eta=D_{\varphi,\Lambda,f}(\eta^-)=HD(k^{-1}_{\varphi,\Lambda,f}(\eta^-)).$$

\begin{remark}
 Given $\eta\in(0,\frac{1}{2}HD(\Lambda)]$ define $\eta^+=\max\{t\in \mathbb{R}:D_{\varphi,\Lambda,f}(t)=\eta\}$. We can ask for a result for $\eta^+$ similar to the result proved for $\eta^-$. However, note that given different $\eta_1,\eta_2\in [0,\frac{1}{2}HD(\Lambda)]$ the intervals $[\eta_1^-,\eta_1^+]$ and $[\eta_2^-,\eta_2^+]$ are disjoint. Then, we can find a countable set $F$ such that for $\eta\in B=(0,\frac{1}{2}HD(\Lambda)] \setminus F$, $\eta^-=\eta^+$. That is, $D_{\varphi,f}|_B$ is injective and for $\eta\in B$
$$\eta=HD(k_{\varphi,\Lambda,f}^{-1}(-\infty,D^{-1}_{\varphi,\Lambda,f}(\eta)])=HD(k_{\varphi,\Lambda,f}^{-1}(D^{-1}_{\varphi,\Lambda,f}(\eta))).$$   
\end{remark}

\subsection{Another presentation for $\mathcal{J}_{\varphi,\Lambda,f}$}

In this subsection, we characterize (modulus some countable subset $\mathcal{F}_{\varphi,\Lambda,f}\subset \mathcal{L}_{\varphi,\Lambda,f}$) the set of elements $t\in \mathcal{L}_{\varphi,\Lambda,f}$ with $D_{\varphi,\Lambda,f}(t)>0$ and such that
\begin{equation}\label{concentration1}
 D_{\varphi,\Lambda,f}(t)=HD(k^{-1}_{\varphi,\Lambda,f}(t)).   
\end{equation}
In fact, we will show that it coincides with the set $\mathcal{J}_{\varphi,\Lambda,f}$ of the subsection \ref{horse}.

Given $\epsilon>0$, we can take $\ell(\epsilon)\in\mathbb{N}$ sufficiently large such that if $\alpha=(a_0, a_{1} \cdots, a_{2\ell(t,\epsilon)})\\ \in \mathcal{A}^{2\ell(\epsilon)+1}$ and $x,y\in R(\alpha;\ell(t,\epsilon))$ then $\abs{f(x)-f(y)}<\epsilon/4$. If $t\in \mathcal{L}_{\varphi,\Lambda,f}$, let 
$$C(t,\epsilon)=\{\alpha=(a_0, a_{1} \cdots, a_{2\ell(t,\epsilon)})\in \mathcal{A}^{2\ell(t,\epsilon)+1}:R(\alpha;\ell(\epsilon))\cap \Lambda_{t+\epsilon/4}\neq \emptyset \}.$$
Define
$$M(t,\epsilon):=M(C(t,\epsilon))=\bigcap \limits_{n \in \mathbb{Z}} \varphi ^{-n}(\bigcup \limits_{\alpha \in C(t,\epsilon)}  R(\alpha;\ell(\epsilon))).$$
Note that by construction, $\Lambda_{t+\epsilon/4}\subset M(t,\epsilon)\subset \Lambda_{t+\epsilon/2}$ and being $M(t,\epsilon)$ a hyperbolic set of finite type (see subsection \ref{tipofinito} for the corresponding definitions and results), it admits a decomposition 
$$M(t,\epsilon)=\bigcup \limits_{x\in \mathcal{X}(t,\epsilon)} \tilde{\Lambda}_x $$
where $\mathcal{X}(t,\epsilon)$ is a finite index set and for $x\in \mathcal{X}(t,\epsilon)$,\ $\tilde{\Lambda}_x$ is a subhorseshoe or a transient set. Now, because of equalities (\ref{transient}) and (\ref{transient2}) we have
\begin{equation}
  R_{\varphi,\Lambda,f}(t)=HD(\Lambda_t)\leq HD(M(t,\epsilon))=\max\limits_{x\in \mathcal{X}(t,\epsilon)} HD(\tilde{\Lambda}_{x})=\max \limits_{\substack{x\in \mathcal{X}(t,\epsilon): \ \tilde{\Lambda}_x \ is\\ subhorseshoe }}HD(\tilde{\Lambda}_{x}).  
\end{equation}
Write
$$\tilde{M}(t,\epsilon)=\bigcup\limits_{\substack{x\in \mathcal{X}(t,\epsilon): \ \tilde{\Lambda}_x \ is\\ subhorseshoe }}\tilde{\Lambda}_x= \bigcup \limits_{x\in \mathcal{I}(t,\epsilon)} \tilde{\Lambda}_x \cup \bigcup \limits_{x\in \mathcal{J}(t,\epsilon)} \tilde{\Lambda}_x$$ where
 $$\mathcal{I}(t,\epsilon)=\{x\in \mathcal{X}(t,\epsilon): \tilde{\Lambda}_x \ \mbox{is a subhorseshoe and}\ \ HD(\tilde{\Lambda}_x)\geq R_{\varphi,\Lambda,f}(t) \}$$
and 
 $$\mathcal{J}(t,\epsilon)=\{x\in \mathcal{X}(t,\epsilon): \tilde{\Lambda}_x \ \mbox{is a subhorseshoe and}\ \ HD(\tilde{\Lambda}_x)< R_{\varphi,\Lambda,f}(t) \}.$$

Also, remember that for any subhorseshoe $\tilde{\Lambda} \subset \Lambda$, being locally maximal, we have 
\begin{equation}\label{omega}
 \bigcup \limits_{y\in \tilde{\Lambda}}W^s(y)=W^s(\tilde{\Lambda})= \{y\in S: \lim \limits_{n \to \infty}d(\varphi^n(y),\tilde{\Lambda})=0\}.   
\end{equation}
In particular, given $x_1\in \mathcal{X}(t,\epsilon)$ we can find $x_2\in \mathcal{X}(t,\epsilon)$ such that $\tilde{\Lambda}_{x_2}$ is a subhorseshoe with $\omega(x)\subset \tilde{\Lambda}_{x_2}$ for every $x\in\tilde{\Lambda}_{x_1}$; and from this and \ref{omega}, it follows that $\ell_{\varphi,f}(x)=\ell_{\varphi,f}(y)$ for some $y\in\tilde{\Lambda}_{x_2}$. 
Using this, we have 
$$\ell_{\varphi,f}(M(t,\epsilon))=\ell_{\varphi,f}(\tilde{M}(t,\epsilon))=\bigcup \limits_{x\in \mathcal{I}(t,\epsilon)} \ell_{\varphi,f}(\tilde{\Lambda}_x)\cup\bigcup \limits_{x\in \mathcal{J}(t,\epsilon)} \ell_{\varphi,f}(\tilde{\Lambda}_x).$$
Moreover, as $k^{-1}_{\varphi,\Lambda,f}(-\infty,t+\epsilon/5)\subset \bigcup\limits_{n\in \mathbb{N}}\psi_u^{-n}(\pi^u(\Lambda_{t+\epsilon/5}))$ (see remark \ref{r1}), we also conclude from (\ref{omega}) that
\begin{equation}\label{kmenos1}
   k^{-1}_{\varphi,\Lambda,f}(t)\subset k^{-1}_{\varphi,\Lambda,f}(-\infty,t+\epsilon/5)\subset \bigcup\limits_{n\in \mathbb{N}}\psi_u^{-n}(\bigcup\limits_{\substack{x\in \mathcal{I}(t,\epsilon)\cup \mathcal{J}(t,\epsilon)}}K^u(\tilde{\Lambda}_x)). 
\end{equation}
We are ready to define the set of nontrivial points in $\mathcal{L}_{\varphi,\Lambda,f}$ such that equation \ref{concentration1} holds:
$$\mathcal{J}^0_{\varphi,\Lambda,f}:=\{t\in \mathcal{L}_{\varphi,\Lambda,f}:\forall \epsilon>0,\ \  (t-\epsilon/4,t+\epsilon/4)\cap \bigcup \limits_{x\in \mathcal{I}(t,\epsilon)} \ell_{\varphi,f}(\tilde{\Lambda}_x)\neq \emptyset \ \ \text{and} \ \ D_{\varphi,\Lambda,f}(t)>0 \}.$$

Observe that, given $t\in \mathcal{L}_{\varphi,\Lambda,f}\setminus \mathcal{J}^0_{\varphi,\Lambda,f}$ with $D_{\varphi,\Lambda,f}(t)>0$, one can find $\epsilon_0>0$ such that 
$$(t-\epsilon_0/4,t+\epsilon_0/4)\cap \bigcup \limits_{x\in \mathcal{I}(t,\epsilon_0)} \ell_{\varphi,f}(\tilde{\Lambda}_x)=\emptyset$$
using (\ref{kmenos1}) and the definition of $k_{\varphi,\Lambda,f}$, we get
$$k^{-1}_{\varphi,\Lambda,f}(t)\subset \bigcup\limits_{n\in \mathbb{N}}\psi_u^{-n}(\bigcup\limits_{\substack{x\in \mathcal{J}(t,\epsilon_0)}}K^u(\tilde{\Lambda}_x)),$$
and then 
$$HD(k^{-1}_{\varphi,\Lambda,f}(t))\leq \max\limits_{x\in\mathcal{J}(t,\epsilon_0)} HD(K^u(\tilde{\Lambda}_x))=\max\limits_{x\in\mathcal{J}(t,\epsilon_0)} \frac{1}{2}HD(\tilde{\Lambda}_x)<\frac{1}{2}R_{\varphi,\Lambda,f}(t)=D_{\varphi,\Lambda,f}(t).$$
We conclude in this case that $D_{\varphi,\Lambda,f}(t)\neq HD(k^{-1}_{\varphi,\Lambda,f}(t)).$

On the other hand, given $t\in \mathcal{J}^0_{\varphi,\Lambda,f}$ and $\epsilon>0$ we can find $x_0\in \mathcal{I}(t,\epsilon)$ and $r_0\in \tilde{\Lambda}_{x_0}$ such that 
$\ell_{\varphi,f}(r_0)\in (t-\epsilon/4,t+\epsilon/4)$. Also, as 
$$\ell_{\varphi,f}(\tilde{\Lambda}_{x_0})\subset f(\tilde{\Lambda}_{x_0})\subset f(M(t,\epsilon))\subset (-\infty,t+\epsilon/2],$$
we can conclude that $t-\epsilon/4<\max f|_{\tilde{\Lambda}_{x_0}}\leq t+\epsilon/2$ and by definition $HD(\tilde{\Lambda}_{x_0})\geq R_{\varphi,\Lambda,f}(t)$.

Define $\mathcal{J}^*_{\varphi,\Lambda,f}=\mathcal{J}^0_{\varphi,\Lambda,f}\setminus \mathcal{F}_{\varphi,\Lambda,f}$, where 
$$\mathcal{F}_{\varphi,\Lambda,f}=\{t\in \mathcal{J}^0_{\varphi,\Lambda,f}:\exists r>0 \ \text{such that}\ (t-r,t)\cap \mathcal{J}^0_{\varphi,\Lambda,f}= \emptyset \}$$
is the enumerable set of points of $\mathcal{J}^0_{\varphi,\Lambda,f}$ isolated on the left. Note that $\mathcal{J}^*_{\varphi,\Lambda,f}\subset \mathcal{L}^{'}_{\varphi,\Lambda,f}$.

Given $t\in \mathcal{J}^*_{\varphi,\Lambda,f}$ we can find a strictly increasing sequence $\{t_n\}_{n\in\mathbb{N}}$ of elements of $\mathcal{J}^0_{\varphi,\Lambda,f}$ converging to $t$ and then, a sequence of subhorseshoes $\{\Lambda_n\}_{n\geq2}$ with the property that
\begin{itemize}
    \item $R_{\varphi,\Lambda,f}(t_{n-1})\leq HD(\Lambda^n)$,
    \item $t_{n-1}<\max f|_{\Lambda^n}<t_n$,
\end{itemize}
Using the arguments of the previous subsections, we can construct a sequence $\{\tilde{\Lambda}^{n_k}\}_{k\in \mathbb{N}}$ of subhorseshoes of $\Lambda$ with the following properties 
\begin{itemize}
    \item $\tilde{\Lambda}^{n_k}\subset \tilde{\Lambda}^{n_{k+1}}$,
    \item $t_{n_{k-1}}<\max f|_{\tilde{\Lambda}^{n_k}}<t_{n_k}$,
    \item $R_{\varphi,\Lambda,f}(t_{n_{k-1}})\leq HD(\tilde{\Lambda}^{n_k})\leq R_{\varphi,\Lambda,f}(t_{n_k})$,
\end{itemize}
and conclude, as before, that in this case $D_{\varphi,\Lambda,f}(t)= HD(k^{-1}_{\varphi,\Lambda,f}(t)).$

Now, the spectral decomposition theorem and the corollary 3.9 of \cite{GC} let us conclude the following proposition
\begin{proposition}
  Given two subhorseshoes $\tilde{\Lambda}_1$ and $\tilde{\Lambda}_2$ of $\Lambda$ such that $\tilde{\Lambda}_1\nsubseteq \tilde{\Lambda}_2$, we have
    $$HD(\tilde{\Lambda}_1)<HD(\tilde{\Lambda}_2).$$  
\end{proposition}
As $\tilde{\Lambda}^{n_k}\nsubseteq \tilde{\Lambda}^{n_{k+1}}$ for $k\in \mathbb{N}$, this proposition let us conclude that
$$R_{\varphi,\Lambda,f}(t_{n_{k-1}})\leq HD(\tilde{\Lambda}^{n_k})<HD(\tilde{\Lambda}^{n_{k+1}})\leq R_{\varphi,\Lambda,f}(t_{n_{k+1}})$$
that is, $D_{\varphi,\Lambda,f}(t_{n_{k-1}})<D_{\varphi,\Lambda,f}(t_{n_{k+1}})$. From this we get that $D_{\varphi,\Lambda,f}|_{\mathcal{J}^*_{\varphi,\Lambda,f}}$ is strictly increasing. Indeed, if $t^{'}\in \mathcal{J}^*_{\varphi,\Lambda,f}$ is less than $t$, consider $k\in\mathbb{N}$ large such that $t^{'}<t_{n_{k-1}}$, therefore $$D_{\varphi,\Lambda,f}(t^{'})\leq D_{\varphi,\Lambda,f}(t_{n_{k-1}})<D_{\varphi,\Lambda,f}(t_{n_{k+1}})\leq D_{\varphi,\Lambda,f}(t).$$

Moreover, in the context of proposition \ref{horseshoes}, given $\eta\in(0,\frac{1}{2}HD(\Lambda)]$ and $\epsilon>0$, there exist some subhorseshoe 
$\tilde{\Lambda}^{\epsilon}\subset M(\eta^-,\epsilon)$ such that $\bigcup\limits_{n\geq0}\tilde{\Lambda}^n\subset \tilde{\Lambda}^{\epsilon}$. As $R_{\varphi,\Lambda,f}(t_n)<HD(\tilde{\Lambda}^n)$, in particular
$$R_{\varphi,\Lambda,f}(\eta^{-})=\sup \limits_{n\geq 0}R_{\varphi,\Lambda,f}(t_n)\leq \sup \limits_{n\geq 0}HD(\tilde{\Lambda}^n)\leq HD(\tilde{\Lambda}^{\epsilon}).$$

Remember that maximums of subhorseshoes are always elements of the Lagrange spectrum which is a closed set (for any subhorseshoe). Then, as $\eta^{-}$ is the limit of the sequence $\{\max f|_{\tilde{\Lambda}^n} \}_{n\in\mathbb{N}}$, one has $\eta^- \in \ell_{\varphi,\Lambda,f}(\tilde{\Lambda}^{\epsilon})$. From this, we conclude that $\mathcal{J}_{\varphi,\Lambda,f}\subset\mathcal{J}^0_{\varphi,\Lambda,f}$  and as  
$\{ D_{\varphi,\Lambda,f}(t_n)^{-} \}_{n\in\mathbb{N}}$ also converges to $\eta^-$, it follows that $\eta^-\notin \mathcal{F}_{\varphi,\Lambda,f}$ and then $\mathcal{J}_{\varphi,\Lambda,f}\subset \mathcal{J}^*_{\varphi,\Lambda,f}$. Finally, as $D_{\varphi,\Lambda,f}(\mathcal{J}_{\varphi,\Lambda,f})=(0,\frac{1}{2}HD(\Lambda)]$ and $D_{\varphi,\Lambda,f}|_{\mathcal{J}^*_{\varphi,\Lambda,f}}$ is injective, we get the equality $\mathcal{J}_{\varphi,\Lambda,f}=\mathcal{J}^*_{\varphi,\Lambda,f}.$ Define also 
$$\tilde{\mathcal{J}}_{\varphi,\Lambda,f}:=\{t\in \mathcal{L}_{\varphi,\Lambda,f}: D_{\varphi,\Lambda,f}(t)>0 \}\setminus \mathcal{J}^0_{\varphi,\Lambda,f}.$$
\subsection{Proof of theorem \ref{T2}}
In this subsection we will work in the setting of Theorem \ref{c1}. Specifically, let $\varphi_0\in \text{Diff}^{2}_{\omega}(S)$ with a mixing horseshoe $\Lambda_0$ with $HD(\Lambda_0)<1$ and $\tilde{\mathcal{U}}\subset \mathcal{U}$ a residual set of some neighbourhood $\mathcal{U}$ of $\varphi_0$ in $\text{Diff}^{2}_{\omega}(S)$ such that $\Lambda_0$ admits a continuation $\Lambda$ with $HD(\Lambda)<1$ for every $\varphi \in \mathcal{U}$ and such that given $\varphi \in \tilde{\mathcal{U}}$ and $f\in\mathcal{R}^r_{\varphi,\Lambda}$, we have the equality $R_{\varphi,\Lambda,f}=L_{\varphi,\Lambda,f}.$

Fix then, $\varphi \in \tilde{\mathcal{U}}$, $f\in\mathcal{R}^r_{\varphi,\Lambda}$ and let $\eta\in (0,\frac{1}{2}HD(\Lambda)]$. By proposition \ref{R} one has
$$\eta^{-}=\min\{t:D_{\varphi,\Lambda,f}(t)=\eta\}=\min\{t:R_{\varphi,\Lambda,f}(t)=2\eta\}=\min\{t:L_{\varphi,\Lambda,f}(t)=2\eta\}.$$
Then, given $\epsilon>0$, because $L_{\varphi,\Lambda,f}(\eta^{-}-\epsilon)<L_{\varphi,\Lambda,f}(\eta^{-})$ we have
\begin{eqnarray*}
  L_{\varphi,\Lambda,f}(\eta^{-}+\epsilon)&=& HD(\mathcal{L}_{\varphi,\Lambda,f}\cap (-\infty,\eta^{-}+\epsilon))\\ &=&
  \max \{HD(\mathcal{L}_{\varphi,\Lambda,f}\cap (-\infty,\eta^{-}-\epsilon]),HD(\mathcal{L}_{\varphi,\Lambda,f}\cap (\eta^{-}-\epsilon,\eta^{-}+\epsilon)) \}\\ &=& \max \{L_{\varphi,\Lambda,f}(\eta^{-}-\epsilon),HD(\mathcal{L}_{\varphi,\Lambda,f}\cap (\eta^{-}-\epsilon,\eta^{-}+\epsilon)) \}\\ &=& HD(\mathcal{L}_{\varphi,\Lambda,f}\cap (\eta^{-}-\epsilon,\eta^{-}+\epsilon)).
\end{eqnarray*}
then, by continuity
$$L_{\varphi,\Lambda,f}(\eta^{-})=\lim\limits_{\epsilon\rightarrow0^+}L_{\varphi,\Lambda,f}(\eta^{-}+\epsilon)=\lim\limits_{\epsilon\rightarrow0^+}HD(\mathcal{L}_{\varphi,\Lambda,f}\cap (\eta^{-}-\epsilon,\eta^{-}+\epsilon))=L^{\text{loc}}_{\varphi,\Lambda,f}(\eta^{-}).$$
That is $L_{\varphi,\Lambda,f}=L^{\text{loc}}_{\varphi,\Lambda,f}$ on $\mathcal{J}_{\varphi,\Lambda,f}$.

On the other hand, given $t\in \tilde{\mathcal{J}}_{\varphi,\Lambda,f}$, one can find $\epsilon_0>0$ such that 
$$(t-\epsilon_0/4,t+\epsilon_0/4)\cap \mathcal{L}_{\varphi,\Lambda,f}\subset \bigcup \limits_{x\in \mathcal{J}(t,\epsilon_0)} \ell_{\varphi,f}(\tilde{\Lambda}_x)$$
and then 
$$HD((t-\epsilon_0/4,t+\epsilon_0/4)\cap \mathcal{L}_{\varphi,\Lambda,f})<HD(\bigcup \limits_{x\in \mathcal{J}(t,\epsilon_0)} \ell_{\varphi,f}(\tilde{\Lambda}_x))\leq HD(\bigcup \limits_{x\in \mathcal{J}(t,\epsilon_0)} \tilde{\Lambda}_x)<R_{\varphi,\Lambda,f}(t),$$
therefore, in this case 
$$L^{\text{loc}}_{\varphi,\Lambda,f}(t)=\lim\limits_{\epsilon\rightarrow0^+}HD(\mathcal{L}_{\varphi,\Lambda,f}\cap (t-\epsilon,t+\epsilon))\leq HD((t-\epsilon_0/4,t+\epsilon_0/4)\cap \mathcal{L}_{\varphi,\Lambda,f})<L_{\varphi,\Lambda,f}(t).$$
This ends the proof of the theorem.
\section{Theorems \ref{T1} and \ref{T2} for the classical Lagrange spectrum}\label{classical}
In this section, we use the dynamical characterization of the classical Lagrange spectrum to prove theorems \ref{T1} and \ref{T2} in this setting. 

Let $N\geq 4$ be an integer. In \cite{LM2} is proved that the portion of the classical Lagrange spectrum $\mathcal{L}$ up to $\sqrt{N^2+4N}$ i.e, $\mathcal{L}\cap (-\infty, \sqrt{N^2+4N}]$ is the dynamically defined Lagrange spectrum $\mathcal{L}_{\varphi.\Lambda(N),f}$ associated with some $\varphi$, $\Lambda(N)$ and $f$.
More specifically, if $C_N=\{x=[0;a_1,a_2,...]: a_i\le N, \forall i\ge 1\}$ and $\tilde{C}_N=\{1,2,...,N\}+C_N$, we set $\Lambda(N)=C_N\times \tilde{C}_N$ and
then consider $\varphi:\Lambda(N) \rightarrow \Lambda(N)$ given by
$$\varphi([0;a_1,a_2,...],[a_0;a_{-1},a_{-2},...])=([0;a_2,a_3,...],[a_1;a_0,a_{-1},...]),$$
that can be extended to a $C^{\infty}$ conservative diffeomorphism on a diffeomorphic copy of the 2-dimensional sphere $\mathbb{S}^2$. Also, the real map is given by $f(x,y)=x+y$. Finally, note that in this context $K^u(\Lambda(N))$ can be identified with $\tilde{C}_N$ and that $k|_{K^u(\Lambda(N))}=k_{\varphi,\Lambda(N),f}.$ 

 \begin{theorem}\label{T3}
   For the continuous and surjective function $D:\mathbb{R} \rightarrow [0,1)$ one can associate a decomposition 
   $$\mathcal{L}\cap (3,\infinity)=\mathcal{J}\cup \mathcal{F}\cup \tilde{\mathcal{J}}$$ 
that satisfies 
\begin{itemize}
    \item $\mathcal{J}\subset \mathcal{L}^{'}$
    \item $D|_{\mathcal{J}}$ is strictly increasing,
    \item $D(\mathcal{J})=(0,1),$
    \item $D(t)=HD(k^{-1}(t))$ for every $t\in \mathcal{J}$,
    \item $\mathcal{F}$ is countable,
    \item $D(t)\neq HD(k^{-1}(t))$ for every $t\in \tilde{\mathcal{J}}$.
\end{itemize}
Indeed, given $\eta\in (0,1)$, define $\eta^-=\min\{t\in \mathbb{R}:D(t)=\eta\}$. Then we can set 
$$\mathcal{J}=\{\eta^-:\eta\in (0,1)\}.$$

\end{theorem}
and we also have the 

\begin{theorem}\label{T4}
If $\mathcal{J}$ and $\tilde{\mathcal{J}}$ are as in theorem \ref{T3}, then one has 
\begin{itemize}
    \item $L^{\text{loc}}(t)=L(t)$ for every $t\in \mathcal{J}$,
    \item $L^{\text{loc}}(t)<L(t)$ for every $t\in \tilde{\mathcal{J}}$.
\end{itemize}
\end{theorem}

Note that for $N_1,N_2\geq 4$ integers if $N_1\leq N_2$ then $\Lambda(N_1)\subset \Lambda(N_2)$. Now, in \cite{M3} was proved for $t\in \mathbb{R}$ that, if $N\in \mathbb{N}$ is arbitrary such that $t< \sqrt{N^2+4N}$ then
\begin{equation}\label{e1}
D(t)=HD(k^{-1}(-\infty,t])=HD(D^u(\Lambda(N)_t))=\frac{1}{2}HD(\Lambda(N)_t)=\frac{1}{2}R_{\varphi,\Lambda(N),f}(t)
\end{equation}
and by proposition \ref{R}, we have 
$$D(t)=D_{\varphi,\Lambda(N),f}(t).$$
Arguing as in the proof of proposition \ref{R}, for $t$ and $N$ as before, it follows easily also that 
$$HD(k^{-1}(t))=HD(k_{\varphi,\Lambda(N),f}^{-1}(t)).$$

We are ready to prove Theorem \ref{T3}. Observe that the sequences of sets $\{\mathcal{J}_{\varphi,\Lambda(N),f}\}_{N\geq4}$, $\{\mathcal{F}_{\varphi,\Lambda(N),f}\}_{N\geq4}$ and $\{\tilde{\mathcal{J}} _{\varphi,\Lambda(N),f}\}_{N\geq4}$  are all increasing and that
$$\mathcal{J}=\{\eta^-:\eta\in (0,1)\}=\bigcup\limits_{N\geq4} \mathcal{J}_{\varphi,\Lambda(N),f}.$$
Clearly, $D(\mathcal{J})=D(\{\eta^-:\eta\in (0,1)\})=(0,1)$ and $D|_{\mathcal{J}}$ is strictly increasing. As for $N\geq 4$, $\mathcal{J}_{\varphi,\Lambda(N),f}\subset \mathcal{L}_{\varphi,\Lambda(N),f}^{'}\subset \mathcal{L}^{'}$ then $\mathcal{J}\subset \mathcal{L}^{'}$. Define also  
$$\mathcal{F}:=\bigcup\limits_{N\geq4}\mathcal{F}_{\varphi,\Lambda(N),f}\ \ \text{and} \ \ \tilde{\mathcal{J}}:=\bigcup\limits_{N\geq4}\tilde{\mathcal{J}}_{\varphi,\Lambda(N),f}$$
Therefore, we have

\begin{eqnarray*}
\mathcal{L}\cap (3,\infinity)=\{t\in \mathcal{L}:D(t)>0\}&=&\bigcup\limits_{N\geq4}\{t\in \mathcal{L}_{\varphi,\Lambda(N),f}:D_{\varphi,\Lambda(N),f}(t)>0\}\\&=& \bigcup\limits_{N\geq4}\mathcal{J}_{\varphi,\Lambda(N),f}\cup \mathcal{F}_{\varphi,\Lambda(N),f}\cup \tilde{\mathcal{J}}_{\varphi,\Lambda(N),f}\\ &=&\mathcal{J}\cup \mathcal{F}\cup \tilde{\mathcal{J}}. 
\end{eqnarray*}

Note that $\mathcal{F}$ is countable because $\mathcal{F}_{\varphi,\Lambda(N),f}$ is countable for every $N\geq 4$. That $D(t)=HD(k^{-1}(t))$ for every $t\in \mathcal{J}$ and $D(t)\neq HD(k^{-1}(t))$ for every $t\in \tilde{\mathcal{J}}$ follow from the corresponding properties for $\mathcal{J}_{\varphi,\Lambda(N),f}$ and $\tilde{\mathcal{J}}_{\varphi,\Lambda(N),f}$ and our previous comments on $D$ and $HD(k^{-1}(\cdot))$. This finishes the proof of the theorem \ref{T3}.


For theorem \ref{T4}, a little bit more has to be done. Given $t\in \mathcal{L}$ such that $D(t)>0$, set $\Lambda=\Lambda(N)$ where $N\geq 4$ is fixed such that $t< \sqrt{N^2+4N}$. First, observe that by (\ref{e2}) and (\ref{e1})
$$L_{\varphi,\Lambda,f}(t)=L(t)=\min \{1,2D(t)\}=\min\{1,R_{\varphi,\Lambda,f}(t)\}.$$

We will consider some cases: Suppose that $t<t_1$, where $t_1:=\sup \{s\in \mathbb{R}:L(s)<1\}=3.334384...$ (see \cite{MMPV}). Then, for $\tilde{t}$ close to $t$ one has $L_{\varphi,\Lambda,f}(\tilde{t})=\min\{1,R_{\varphi,\Lambda,f}(\tilde{t})\}=R_{\varphi,\Lambda,f}(\tilde{t})$.

If we suppose also that $t\in \tilde{\mathcal{J}}$, the same proof as before let us conclude that 
$$L^{\text{loc}}(t)=L^{\text{loc}}_{\varphi,\Lambda,f}(t)<R_{\varphi,\Lambda,f}(t)=L_{\varphi,\Lambda,f}(t)=L(t).$$

If $t\in \mathcal{J}$, then 
\begin{eqnarray*}
  t=\min\{\tilde{t}:R_{\varphi,\Lambda,f}(\tilde{t})=L(t)\}&=&\min\{\tilde{t}:\min\{1,R_{\varphi,\Lambda,f}(\tilde{t})\}=L(t)\}\\ &=&\min\{\tilde{t}:L_{\varphi,\Lambda,f}(\tilde{t})= L(t)\} \end{eqnarray*}
and proceeding as in the previous subsection, for small $\epsilon>0$ we conclude that
\begin{eqnarray*}
L(t+\epsilon)=L_{\varphi,\Lambda,f}(t+\epsilon)&=&HD(\mathcal{L}_{\varphi,\Lambda,f}\cap (t-\epsilon,t+\epsilon))\\ &=&HD(\mathcal{L}\cap (t-\epsilon,t+\epsilon))
\end{eqnarray*}
and then
$L(t)=L^{\text{loc}}(t).$ Note that the same argument also works for $t=t_1$ and then one also has $L^{\text{loc}}(t_1)=L(t_1).$

Now, suppose that $t>t_1$ and take $\epsilon>0$ small. If $\tilde{M}(t,\epsilon)$ is as before and $\xi$ is the fixed orbit given by the kneading sequence $(1)_{i\in \mathbb{Z}}$, we can consider the alternative decomposition 

\begin{equation}\label{M}
\tilde{M}(t,\epsilon)=\bigcup\limits_{\substack{x\in \mathcal{X}(t,\epsilon): \ \tilde{\Lambda}_x \ is\\ subhorseshoe }}\tilde{\Lambda}_x= \bigcup \limits_{i\in \tilde{\mathcal{I}}(t,\epsilon)} \tilde{\Lambda}_i \cup \bigcup \limits_{i\in \tilde{\mathcal{J}}(t,\epsilon)} \tilde{\Lambda}_j
\end{equation}
where
 $$\tilde{\mathcal{I}}(t,\epsilon)=\{i\in \mathcal{X}(t,\epsilon): \tilde{\Lambda}_i \ \mbox{is a subhorseshoe and it connects with}\ \xi\ \mbox{before}\ t+\epsilon \}$$
and 
 $$\tilde{\mathcal{J}}(t,\epsilon)=\{j\in \mathcal{X}(t,\epsilon): \tilde{\Lambda}_j \ \mbox{is a subhorseshoe and it does not connect with}\ \xi\ \mbox{before}\ t+\epsilon \}.$$

In \cite{GC} is showed that for any $j\in\tilde{\mathcal{J}}(t,\epsilon)$ one has $HD(\tilde{\Lambda}_j)<0.99$. As, in our present case
$$1=L_{\varphi,\Lambda,f}(t_1)=L_{\varphi,\Lambda,f}(t)=\min\{1,R_{\varphi,\Lambda,f}(t)\},$$
one has $R_{\varphi,\Lambda,f}(t)\geq 1$ and then $HD(\tilde{M}(t,\epsilon))\geq R_{\varphi,\Lambda,f}(t)\geq 1$, that let us conclude that $\mathcal{I}(t,\epsilon)\neq \emptyset$ and $\mathcal{I}(t,\epsilon)\subset \tilde{\mathcal{I}}(t,\epsilon)$. 

Suppose $t\in \tilde{\mathcal{J}}$, then one can find $\epsilon_0>0$ such that for $0<\tilde{\epsilon}\leq \epsilon_0$ one has
$$(t-\tilde{\epsilon}/4,t+\tilde{\epsilon}/4)\cap\bigcup \limits_{i\in \mathcal{I}(t,\tilde{\epsilon})} \ell_{\varphi,f}(\tilde{\Lambda}_i)=\emptyset$$
and 
$$(t-\tilde{\epsilon}/4,t+\tilde{\epsilon}/4)\cap \mathcal{L}= (t-\tilde{\epsilon}/4,t+ \tilde{\epsilon}/4)\cap \mathcal{L}_{\varphi,\Lambda,f}\subset \bigcup \limits_{j\in \mathcal{J}(t,\tilde{\epsilon})} \ell_{\varphi,f}(\tilde{\Lambda}_j).$$


Fix $0<\tilde{\epsilon}<\epsilon_0/4$ and $x\in \mathcal{I}(t,\tilde{\epsilon})$. Consider $j_0\in \mathcal{J}(t,\tilde{\epsilon})$ such that $(t-\tilde{\epsilon}/4,t+\tilde{\epsilon}/4)\cap\ell_{\varphi,f}(\tilde{\Lambda}_{j_0})\neq \emptyset$ and suppose $j_0\in\tilde{\mathcal{I}}(t,\tilde{\epsilon})$. As $\tilde{\Lambda}_{j_0}$ and $\tilde{\Lambda}_x$ are subhorseshoes that connect with $\xi$ before $t+\tilde{\epsilon}$, then by corollary \ref{connection3} there exist some subhorseshoe $\tilde{\Lambda}$ such that $\tilde{\Lambda}_{j_0}\cup\tilde{\Lambda}_x \subset\tilde{\Lambda}\subset \Lambda_{t+\epsilon_0/4}$. But this is a contradiction because in this case, $ HD(\tilde{\Lambda})\geq HD(\tilde{\Lambda}_x)\geq R_{\varphi,\Lambda,f}(t)$ and then, one can find some $i_0\in \mathcal{I}(t,\epsilon_0)$ such that $\tilde{\Lambda}\subset\tilde{\Lambda}_{i_0}$ and

$$\emptyset\neq(t-\tilde{\epsilon}/4,t+ \tilde{\epsilon}/4)\cap \ell_{\varphi,f}(\tilde{\Lambda}_{j_0})\subset(t-\epsilon_0/4,t+ \epsilon_0/4)\cap \ell_{\varphi,f}(\tilde{\Lambda}_{i_0}).$$
It follows that 
$$\bar{\mathcal{J}}(t,\tilde{\epsilon})=\{j\in \mathcal{J}(t,\tilde{\epsilon}):(t-\tilde{\epsilon}/4,t+\tilde{\epsilon}/4)\cap\ell_{\varphi,f}(\tilde{\Lambda}_j)\neq \emptyset\}\subset\tilde{\mathcal{J}}(t,\tilde{\epsilon})$$
and then
$$HD((t-\tilde{\epsilon}/4,t+\tilde{\epsilon}/4)\cap \mathcal{L}_{\varphi,\Lambda,f})<HD(\bigcup \limits_{x\in \bar{\mathcal{J}}(t,\tilde{\epsilon})} \tilde{\Lambda}_x)<0.99$$
that let us conclude that
\begin{eqnarray*}
    L^{\text{loc}}(t)= L^{\text{loc}}_{\varphi,\Lambda,f}(t)&=&\lim\limits_{\epsilon\rightarrow0^+}HD(\mathcal{L}_{\varphi,\Lambda,f}\cap (t-\epsilon,t+\epsilon))\\&\leq& HD((t-\tilde{\epsilon}/4,t+\tilde{\epsilon}/4)\cap \mathcal{L}_{\varphi,\Lambda,f})\\ &<&1=L_{\varphi,\Lambda,f}(t)=L(t).
\end{eqnarray*}
Now, let us come back to (\ref{M}). Using corollary \ref{connection3} at most $\abs{\tilde{\mathcal{I}}(t,\epsilon)}-1$ times, we see that there exists a subhorseshoe $\tilde{\Lambda}(t,\epsilon)\subset \Lambda$ and some $q(t,\epsilon)<t+\epsilon$ such that 
$$\bigcup \limits_{i\in \tilde{\mathcal{I}}(t,\epsilon)} \tilde{\Lambda}_i\subset \tilde{\Lambda}(t,\epsilon)\subset \Lambda_{q(t,\epsilon)}.$$
Note also that
$$R_{\varphi,\Lambda,f}(t+\epsilon/4)\leq HD(M(t,\epsilon))=HD(\tilde{M}(t,\epsilon))\leq HD(\tilde{\Lambda}(t,\epsilon))\leq  R_{\varphi,\Lambda,f}(q(t,\epsilon)).$$
Suppose $t\in \mathcal{J}$, then because $t$ is accumulated on the left by points of $\mathcal{J}$, one has for small $\epsilon$ 
$$1< HD(\tilde{\Lambda}(t-\epsilon,\epsilon/4))\leq R_{\varphi,\Lambda,f}(q(t-\epsilon,\epsilon/4))<R_{\varphi,\Lambda,f}(t)\leq HD(\tilde{\Lambda}(t,\epsilon)).$$
Observe that in this situation 
$$\max f|_{\tilde{\Lambda}(t-\epsilon,\epsilon/4) }<t-3\epsilon/4<t \leq \max f|_{\tilde{\Lambda}(t,\epsilon)}.$$ 

Take some point $x_{\epsilon}\in \tilde{\Lambda}(t,\epsilon)$ such that $f(x_{\epsilon})=\max f|_{\tilde{\Lambda}(t,\epsilon)}$ and $n_0\in \mathbb{N}$ sufficiently large such that for any $x$ and $y$ if their kneading sequences coincide in the central block (centered at the zero position) of size $2n_0+1$ then $\abs{f(x)-f(y)}<\epsilon/4$. If the kneading sequence of $x_{\epsilon}$ is $\{x_n\}_{n\in \mathbb{Z}}$ and $\alpha=(x_{-n_0},\dots,x_0, \dots, x_{n_0})$ then one has 
$$\max f|_{\tilde{\Lambda}(t-\epsilon,\epsilon/4)}< t-3\epsilon/4 < t-\epsilon/4 \leq \min f|_{R(\alpha;0)}.$$ 

Now, as $\tilde{\Lambda}(t,\epsilon)$ is transitive and contains the periodic fixed orbit $\xi$, we can find some $n_1>4n_0+1$ and some words $\alpha_1$ and $\alpha_2$ such that  
$$\tilde{\alpha}=(\tilde{\alpha}_{-n_1},\dots,\tilde{\alpha}_0,\dots, \tilde{\alpha}_{n_1})=(\underbrace{1, \dots ,1}_{n_0+1 \ times},\alpha_1,\alpha,\alpha_2,\underbrace{1, \dots ,1}_{n_0 \ times})$$
appears in the kneading sequence of some point in $\tilde{\Lambda}(t,\epsilon)$ and $\alpha=(\tilde{\alpha}_{-n_0},\dots,\tilde{\alpha}_0,\dots, \tilde{\alpha}_{n_0}).$

Consider the neighborhood $P_{t,\epsilon}=R(\beta;0)$ of $(1)_{i\in \mathbb{Z}}$ where $\beta=(\beta_{-n_0},\dots,\beta_{n_0})=(1,\dots,1)$ and define on $P_{t,\epsilon}$ the diffeomorphism  
$$\Phi([0;a_1,a_2,...],[a_0;a_{-1},a_{-2},...])=([0;\tilde{\alpha}_1,\dots,\tilde{\alpha}_{n_1},a_1,a_2,...],[\tilde{\alpha}_0;\tilde{\alpha}_{-1},\dots,\tilde{\alpha}_{-n_1},a_0,a_{-1},...]),$$
which in sequences corresponds to $(\eta^{-};\eta_0,\eta^{+}) \rightarrow (\eta^{-},\eta_0,\tilde{\alpha},\eta^{+})$, where the zero position of $(\eta^{-},\eta_0,\tilde{\alpha},\eta^{+})$ is the zero position of $\tilde{\alpha}.$

Note that, by construction, for every $x\in P_{t,\epsilon}\cap \tilde{\Lambda}(t-\epsilon,\epsilon/4)$ one has $m_{\varphi,\Lambda,f}(\Phi(x))=f(\varphi^n(\Phi(x)))$, where $n\in \{-n_1,\dots,n_1\}$ corresponds to some position of the word $\tilde{\alpha}$. Write
 $$P_{t,\epsilon}=\bigcup_{n=-n_1}^{n_1} P_{t,\epsilon,n}$$
 where 
 $$P_{t,\epsilon,n}=\{x\in P_{t,\epsilon}\cap \tilde{\Lambda}(t-\epsilon,\epsilon/4):m_{\varphi,\Lambda,f}(\Phi(x))=f(\varphi^n(\Phi(x))) \}$$ 
 is a closed subset and then one can find some $\tilde{n}\in \{-n_1,\dots,n_1\}$ such that $P_{t,\epsilon,\tilde{n}}$ contain the intersection of some rectangle $B_{t,\epsilon,\tilde{n}}$ of some Markov partition of $\tilde{\Lambda}(t-\epsilon,\epsilon/4)$ with $\tilde{\Lambda}(t-\epsilon,\epsilon/4)$. By proposition $1$ of \cite{GCD} we can find some complete subshift 
$\Sigma(\mathcal{B})$, associated to a finite set $\mathcal{B}$ of big words in the alphabet $\mathcal{A}(\tilde{\Lambda}(t-\epsilon,\epsilon/4))$ of $\tilde{\Lambda}(t-\epsilon,\epsilon/4)$, and then in the alphabet $\{1,2,...,N\}$, such that the subhorseshoe $\Lambda(\Sigma(\mathcal{B}))=\Pi^{-1}(\bigcup_{i\in \mathbb{Z}} \sigma ^i(\Sigma(\mathcal{B})))$ determined by $\Sigma(\mathcal{B})$ has Hausdorff dimension close to $HD(\tilde{\Lambda}(t-\epsilon,\epsilon/4))>1$ and is contained in $\tilde{\Lambda}(t-\epsilon,\epsilon/4)$.

As Gauss-Cantor are non-essentially affine and because of the transitivity of $\tilde{\Lambda}(t-\epsilon,\epsilon/4)$, we conclude that  $B_{t,\epsilon,\tilde{n}}\cap \tilde{\Lambda}(t-\epsilon,\epsilon/4)$ contains the product of two non-essentially affine Cantor sets $K_1$ and $K_2$ with $HD(K_1\times K_2)=HD(\Lambda(\Sigma(\mathcal{B})))$. Note also that $g=f\circ\varphi^{\tilde{n}}\circ \Phi$ has non-zero partial derivatives in $P_{t,\epsilon}$ and then we can apply theorem \ref{formula} and conclude that 
$$1=\min \{1, HD(\Lambda(\Sigma(\mathcal{B})))\}=\min \{1, HD(K_1\times K_2) \}=HD(g(K_1\times K_2))$$
therefore 
\begin{eqnarray*}
 1=HD(g(K_1\times K_2))&=&HD(g(B_{t,\epsilon,\tilde{n}}\cap \tilde{\Lambda}(t-\epsilon,\epsilon/4)))\leq HD(m_{\varphi,\Lambda,f}(\Phi(P_{t,\epsilon,\tilde{n}})))\\ &\leq& HD(\mathcal{M}_{\varphi,\Lambda,f}\cap (t-5\epsilon/4,t+5\epsilon/4))\leq 1. \end{eqnarray*}

Finally, as was proved in \cite{MM}, $HD(\mathcal{M}\setminus \mathcal{L})<1$, then one conclude that for small $\epsilon$
$$1=HD(\mathcal{M}_{\varphi,\Lambda,f}\cap (t-\epsilon,t+\epsilon))=\max \{HD(\mathcal{L}_{\varphi,\Lambda,f}\cap (t-\epsilon,t+\epsilon)),HD((\mathcal{M}\setminus \mathcal{L})\cap (t-\epsilon,t+\epsilon))\}$$
$$\ \ \ \  =HD(\mathcal{L}_{\varphi,\Lambda,f}\cap (t-\epsilon,t+\epsilon))$$
and then in this case also holds that 
$$L^{\text{loc}}(t)=L^{\text{loc}}_{\varphi,\Lambda,f}(t)=\lim\limits_{\epsilon\rightarrow0^+}HD(\mathcal{L}_{\varphi,\Lambda,f}\cap (t-\epsilon,t+\epsilon))=1=L_{\varphi,\Lambda,f}(t)=L(t).$$

\begin{corollary}
If $L^{\text{loc}}|_{\mathcal{L}^{'}}$ is non-decreasing, then $\tilde{\mathcal{J}}\cap 
 \mathcal{L}^{'}=\emptyset$.   
\end{corollary}
\begin{proof}
 
Suppose $t\in \tilde{\mathcal{J}}\cap 
 \mathcal{L}^{'}$. In particular,
$$D(t)^-=\min\{s\in \mathbb{R}:D(s)=D(t)\}<t$$
and
$$L(D(t)^-)=\min\{1,2D(D(t)^-)\}=\min\{1,2D(t)\}=L(t).$$
Finally, it follows from theorem \ref{T4} that 
$$L^{\text{loc}}(t)<L(t)=L(D(t)^-)=L^{\text{loc}}(D(t)^-).$$
\end{proof}

It is convenient to point here that Theorem \ref{T3} does not imply the main theorem of \cite{GC} because our result does not allow us to conclude, for example, that $D|_{T}$ is injective, where $T$ is the interior of the Lagrange spectrum. On the other hand, as $\max f|_{\Lambda(4)}=\sqrt{32}=5.65685$ and $(5,\infty)\subset T$ (see \cite{F}), one has
$$\mathcal{F}=\bigcup\limits_{N\geq4}\mathcal{F}_{\varphi,\Lambda(N),f}=\mathcal{F}_{\varphi,\Lambda(4),f}.$$

\subsection*{Acknowledgements:} I would like to thank Davi Lima and Carlos Moreira for helpful comments and suggestions which substantially improved this work.


\begin{thebibliography}{10}

\bibitem{Ar} P. Arnoux, 
\emph{Le codage du flot g\'eod\'esique sur la surface modulaire}, Enseign. Math. (2) 40, no. 1-2, 1994, 29-48.

\bibitem{LM2} D. Lima and C. G. Moreira, \emph{Dynamical characterization of initials segments of the Markov and Lagrange spectra}

\bibitem{PES} Y. Pesin, Dimension Theory in Dynamical Systems: Contemporary Views and Applications, University of Chicago Press, 1998.

\bibitem{MMPV} C. Matheus, C. G. Moreira, M. Pollicott and P. Vytnova, \emph{Hausdorff dimension of Gauss--Cantor sets and two applications to classical Lagrange and Markov spectra}. https://arxiv.org/abs/2106.06572

\bibitem{CMM16}
A.~Cerqueira, C. ~Matheus, C. G.~Moreira.
\newblock Continuity of Hausdorff dimension across generic dynamical Lagrange
  and Markov spectra.
\newblock {\em Journal of Modern Dynamics}, 12:151--174, 2018.
 
\bibitem{LMMR}
D.~Lima, C. Matheus, C. G. Moreira and S. Roma\~na 
\newblock   {\em Classical and Dynamical Markov and Lagrange spectra},
\newblock{\em World Scientific}, 2020.

\bibitem{F} G.A. Freiman, \emph{Diophantine approximation and geometry of numbers (The Markoff spectrum)}, Kalininskii Gosudarstvennyi Universitet, Moscow, 1975.

\bibitem{LM}
D.~Lima and C. G. Moreira.
\newblock  Phase transtitions on the Markov and Lagrange dynamical spectra.
\newblock{\em Annales de L'Institute Henri Poincar\'e (C), Analyse non-lineaire}, 1--31,2020.

\bibitem{BK}
B. P. ~Kitchens.
\newblock {\em Symbolic Dynamics: One-sided, Two-sided and Countable State Markov Shifts},
\newblock Universitext, Springer, 1997.

\bibitem{MMan}
H. ~McCluskey and A. Manning.
\newblock Hausdorff dimension for horseshoes.
\newblock {\em Ergodic Theory and Dynamical Systems}, 3:251--260, 1983.

\bibitem{Shub} M. Shub. \emph{Global Stability of Dinamical Systems}. Springer-Verlag, 1986.

\bibitem{M79}
A.~Markoff.
\newblock Sur les formes quadratiques binaires ind\'efinies.
\newblock {\em Math.Ann.}, 15:381--406, 1879.

\bibitem{GCD}
C. G. Moreira, C. Villamil and D. Lima.
\newblock Continuity of fractal dimensions in conservative generic Markov and Lagrange dynamical spectra,
\newblock Preprint (2023) available at arXiv:2305.07819

\bibitem{GC}
C. G. Moreira and C. Villamil.
\newblock Concentration of dimension in extremal points of left-half lines in the Lagrange spectrum,
\newblock Preprint (2023) available at arXiv:2309.14646

\bibitem{C1}
C. Villamil.
\newblock On the discontinuities of Hausdorff dimension in generic dynamical Lagrange spectrum,
\newblock Preprint (2024) available at arXiv:2403.18940

\bibitem{M3}
C.~G. Moreira.
\newblock Geometric properties of Markov and Lagrange spectra.
\newblock {\em Annals of Math.}, 188: 145--170, 2018.

\bibitem{P} O. Perron, \emph{\"Uber die Approximation irrationaler Zahlen durch rationale
II}, S.-B. Heidelberg Akad. Wiss., Abh. 8, 1921, 12 pp.

\bibitem{MR2}
C.~G. Moreira and S.~Roma\~na.
\newblock On the Lagrange and Markov dynamical spectra.
\emph{Ergodic Theory and Dynamical Systems}, Volume 37, Issue 5, August 2017, pp. 1570 - 1591

\bibitem{CF}
T.~Cusick and M.~Flahive, 
\newblock{The Markoff and Lagrange spectra}, 
\emph{Mathematical Surveys and Monographs}, \textbf{30}. American Mathematical Society, Providence, RI, 1989. x+97 pp.

\bibitem{MY-10}
C.~G. Moreira and J.-C. Yoccoz.
\newblock Tangences homoclines stables pour des ensembles hyperboliques de
  grande dimension fractale.
\newblock {\em Annales Scientifiques de l'\'Ecole Normale Sup\'erieure},
  43:1--68, 2010.

\bibitem{M50}
C. G. Moreira.
\newblock Geometric properties of images of cartesian products of regular Cantor
sets by differentiable real maps,
\newblock {\em Mathematische Zeitschrift}, (2023) 303:3. doi:10.1007/s00209-022-03151-z

\bibitem{PT93}
J.~Palis and F.~Takens.
\newblock {\em Hyperbolicity and Sensitive chaotic dynamics at homoclinic
  biifurcations: fractal dimensios and infinitely many attractors}.
\newblock Cambridge Univ. Press, 1993.

\bibitem{S.ITO} S.~Ito,
\newblock  Number Theoretic expansions, Algorithms and Metrical observations.
\newblock{\em S\'eminaire de Th\'eorie des Nombres de Bordeaux}, 1--27,1984.

\bibitem{MM} Matheus, C., Moreira, C.G.: Fractal geometry of the complement of Lagrange spectrum in Markov spectrum. \emph{Comment. Math. Helv.},   \textbf{95}, no. 3, 593--633 (2020)

\end{thebibliography}
\end{document}